\documentclass{amsart}
\usepackage{amssymb,mathtools}
\usepackage[british]{babel}
\usepackage{enumitem}
\usepackage[latin1]{inputenc}

\usepackage{hyperref}

\usepackage{tikz-cd}

\newtheorem{theorem}{Theorem}
\numberwithin{theorem}{section}
\newtheorem{corollary}[theorem]{Corollary}
\newtheorem{lemma}[theorem]{Lemma}
\newtheorem{proposition}[theorem]{Proposition}
\theoremstyle{definition}
\newtheorem{definition}[theorem]{Definition}
\newtheorem{remark}[theorem]{Remark}

\numberwithin{equation}{section}

\newcommand{\atrs}{\mathsf{ATR_0^{set}}}
\newcommand{\prs}{\mathsf{PRS}\omega}
\newcommand{\otp}{\mathsf{otp}}
\newcommand{\lo}{\mathsf{LO}}
\newcommand{\nat}{\mathsf{Nat}}
\newcommand{\supp}{\operatorname{supp}}
\newcommand{\rng}{\operatorname{rng}}
\newcommand{\en}{\operatorname{en}}
\newcommand{\tr}{\operatorname{Tr}}

\title[Induction on Dilators and Bachmann-Howard Fixed Points]{Induction on Dilators and\\ Bachmann-Howard Fixed Points}

\author[J.~Aguilera]{Juan P.\ Aguilera}
\author[A.~Freund]{Anton Freund}
\author[A.~Weiermann]{Andreas Weiermann}

\address{Juan P.\ Aguilera, Institut f\"ur diskrete Mathematik und Geometrie, Technische Universit\"at Wien, Wiedner Hauptstrasse 8-10, 1040 Wien, Austria}
\email{aguilera@logic.at}
\address{Anton Freund, University of W\"urzburg, Institute of Mathematics, Emil-Fischer-Str.~40, 97074 W\"urz\-burg, Germany}
\email{anton.freund@uni-wuerzburg.de}
\address{Andreas Weiermann, Department of Mathematics: Analysis, Logic and Discrete Mathematics, Ghent University, Krijgslaan 281 S8, 9000 Ghent, Belgium}
\email{andreas.weiermann@ugent.be}

\thanks{Work on the present paper has been supported by FWF grant STA-139 (Aguilera) as well as DFG project 460597863~(Freund).}

\begin{document}

\begin{abstract}
One of the most important principles of J.-Y.~Girard's $\Pi^1_2$-logic is induction on dilators. In particular, Girard used this principle to construct his famous functor~$\Lambda$. He claimed that the totality of $\Lambda$ is equivalent to the set existence axiom of $\Pi^1_1$-comprehension from reverse mathematics. While Girard provided a plausible description of a proof around 1980, it seems that the very technical details have not been worked out to this day. A few years ago, a loosely related approach led to an equivalence between $\Pi^1_1$-comprehension and a certain Bachmann-Howard principle. The present paper closes the circle. We relate the Bachmann-Howard principle to induction on dilators. This allows us to show that $\Pi^1_1$-comprehension is equivalent to the totality of a functor~$\mathbb J$ due to P.~P\"apping\-haus, which can be seen as a streamlined version of~$\Lambda$.
\end{abstract}


\maketitle

\section{Introduction}

It is known that the notion of well order is $\Pi^1_1$-complete: given a $\Pi^1_1$-statement~$\varphi$, we can compute a countable linear order that is well-founded precisely if~$\varphi$ holds, provably in a weak theory. Analogously, $\Pi^1_2$-statements can be characterized by transformations of well orders. In fact, it suffices to consider very uniform transformations called dilators, which have been singled out by J.-Y.~Girard~\cite{girard-pi2}. Specifically, a dilator is an endofunctor on the category of well orders and order embeddings that preserves direct limits and pullbacks (see Section~\ref{sect:prelim-dil} for detailed explanations). Dilators are determined by their values on finite orders and embeddings. Thus they are essentially set-sized objects, even though the category of well orders is a proper class. Dilators that map finite arguments to countable values are essentially countable, so that we can ask whether they are computable. We can now make the above analogy more precise: given a $\Pi^1_2$-statement~$\psi$, we can compute an endofunctor~$D$ of linear orders such that $\psi$ holds precisely if~$D$ is a dilator, i.\,e., if $D(X)$ is well-founded for any well order~$X$ (see~\cite[Annex~8.E]{girard-book-part2} for D.~Normann's proof of this result of Girard). It turns out that many natural $\Pi^1_2$-statements (e.\,g., from reverse mathematics) correspond to natural dilators, which arise from classical notation systems studied in ordinal analysis (see the survey by M.~Rathjen~\cite{rathjen-axiomatic-thinking}).

To characterize a genuine $\Pi^1_3$-statement -- like the $\Pi^1_1$-comprehension axiom from reverse mathematics -- one needs to move up another type-level. More explicitly, one must now consider transformations that take a dilator as input and produce either another dilator or a well order as output. In the following, the so-called Bachmann-Howard fixed point~$\vartheta D$ is a linear order that can be computed from a given countable dilator~$D$ (note that only countable dilators occur in the setting of reverse mathematics).

\begin{theorem}[{\cite{freund-equivalence,freund-computable}}]\label{thm:theta}
The following are equivalent over~$\mathsf{RCA}_0$:
\begin{enumerate}[label=(\roman*)]\label{thm:bachmann}
\item $\Pi^1_1$-comprehension,
\item for any dilator~$D$, the order~$\vartheta D$ is well-founded.
\end{enumerate}
\end{theorem} 

The proof of the theorem follows an approach that Rathjen~\cite{rathjen-atr} suggested in a paper published in 2014. More than 30 years earlier, Girard~\cite{girard-book-part2} had sketched a similar result. While his approach is plausible, it is also very technical and was apparently never worked out in full detail. Even the statement of Girard's result in the setting of reverse mathematics is not entirely unproblematic. Indeed, the result involves a certain functor from dilators to dilators, which is denoted by~$\Lambda$. Roughly speaking, this functor embodies an iteration of binary exponentiation along dilators. However, the precise definition of $\Lambda$ is complicated even in a set-theoretic setting. It is reasonable to expect that $\Lambda$ can be defined in second-order arithmetic, but it seems that this was never made explicit.

In the present paper, we show that $\Pi^1_1$-comprehension is equivalent to the totality of a functor~$\mathbb J$ due to P.~P\"appinghaus~\cite{paeppinghaus-bachmann}. This functor can be seen as a streamlined version of~$\Lambda$, which avoids some technicalities while the conceptual challenge remains the same. To further facilitate the presentation, we prove the equivalence not over second-order arithmetic but in a set theory~$\atrs$ due to S.~Simpson~\cite{simpson82}. This theory is conservative over the system $\mathsf{ATR}_0$ from reverse mathematics, so that it is weak enough to make a characterization of~$\Pi^1_1$-comprehension informative. The definition of $\atrs$ will be recalled in Section~\ref{sect:prs}. For now, we just mention that we include the axiom of countability (`all sets are countable'), which~\cite{simpson82} describes as optional. This axiom seems natural from the viewpoint of second-order arithmetic. It simplifies the formulation of our results but can be avoided in favour of explicit countability assumptions.

To describe the definition of~$\mathbb J$, we recall that Girard has classified dilators into four kinds. As a first approximation (see Section~\ref{sect:prelim-dil} for the official definition), we focus on countable dilators (say in a model of $\atrs$) but assume that we have access (outside of such a model) to the least uncountable cardinal, which we denote by~$\Omega$. Then each dilator~$D$ validates precisely one of the following:
\begin{enumerate}
\item We have $D(\Omega)=0$, and then $D$ is constant zero.
\item The ordinal $D(\Omega)$ is a successor, and we get $D=D'+1$ for a dilator~$D'$.
\item Our value $D(\Omega)$ is a limit ordinal of countable cofinality, and then we can write $D=\sum_{\alpha<\lambda}D_\alpha$ for a countable limit~$\lambda$ and dilators~$D_\alpha\neq 0$.
\item The ordinal~$D(\Omega)$ is a limit of uncountable cofinality, and it follows that we have $D(\Omega)=\sup_{\alpha<\Omega}\{D\}^\alpha(\Omega)$ for certain dilators~$\{D\}^\alpha$.
\end{enumerate}
The sum in~(3) is unique if we demand that the $D_\alpha$ are connected, i.\,e., cannot be writen as sums themselves. When~(4) applies, we say that $D$ is of type~$\Omega$. In each case, $D$ is approximated by predecessors with a smaller value on~$\Omega$ (the dilators~$D'$ and $\sum_{\alpha<\nu}D_\alpha$ for~$\nu<\lambda$ as well as $\{D\}^\alpha$ for~$\alpha<\Omega$). This suggests principles of induction and recursion over dilators, which were identified by Girard~\cite{girard-pi2}.

Now the functor~$\mathbb J$ of P\"appinghaus (see~\cite{paeppinghaus-bachmann} with $g=1$) maps a dilator~$D$ and an ordinal~$\gamma$ to the value~$\mathbb J(D,\gamma)$ that is recursively explained by
\begingroup
\allowdisplaybreaks
\begin{align*}
\mathbb J(0,\gamma)&=\gamma,\\
\mathbb J(D+1,\gamma)&=\mathbb J(D,\gamma)+1,\\
\mathbb J\left(\textstyle\sum_{\alpha<\lambda}D_\alpha,\gamma\right)&=\textstyle\sup_{\nu<\lambda}\mathbb J\left(\textstyle\sum_{\alpha<\nu}D_\alpha,\gamma\right)\!,\\
\mathbb J(D,\gamma)&=\alpha+\mathbb J\left(\{D\}^\alpha,\gamma\right)\text{ with }\alpha=\mathbb J\left(\{D\}^0,\gamma\right)\!.
\end{align*}
\endgroup
In the third of these clauses, it is understood that~$\lambda$ is a limit ordinal and that the dilators $D_\alpha$ are connected. The fourth clause applies when~$D$ is of type~$\Omega$.

As mentioned above, we want to show that $\Pi^1_1$-comprehension is equivalent to the statement that $\mathbb J$ is total, provably in~$\atrs$. For this purpose, we must first define~$\mathbb J$ in~$\atrs$ without constructing it as a total function. Let us draw an analogy with the classical result that $\Pi_2$-reflection over Peano arithmetic~($\mathsf{PA}$) is equivalent to the totality of certain functions~$F_{\alpha}:\mathbb N\to\mathbb N$ for all~$\alpha<\varepsilon_0$, provably in~$\mathsf{PA}$. The most straightforward definition of~$F_\alpha$ is by transfinite recursion on~$\alpha$. While the latter is not available in~$\mathsf{PA}$, it is known that the relation $F_\alpha(x)=y$ (not the function~$(\alpha,x)\mapsto F_\alpha(x)$) has a primitive recursive definition for which~$\mathsf{PA}$ proves `local correctness'. More explicitly, $\mathsf{PA}$ proves that the clauses of the transfinite recursion hold whenever the relevant values of~$F_\alpha$ are defined (see Section~5.2 of~\cite{sommer95}). We take this analogy as guidance for a definition of~$\mathbb J$ in the theory~$\atrs$.

In the following, `primitive recursive' always refers to the primitive recursive set functions of R.~Jensen and C.~Karp~\cite{jensen-karp}, which are available in~$\atrs$ (see Section~\ref{sect:prs}). While the general principle of recursion over dilators goes beyond~$\atrs$ (e.\,g., it yields~$\mathbb J$ as a total function), we will see that primitive recursion is enough to justify a certain form of `guarded' recursion over dilators. Using the latter, we obtain a locally correct definition of~$\mathbb J$ in the form of a primitive recursive relation~$\mathbb J(D,\gamma)\simeq\eta$ (see Sections~\ref{sect:prs} to~\ref{sect:J}). Independently of the application to~$\mathbb J$, we view guarded induction and recursion over dilators as an important contribution of the present paper.

Now that the definition of~$\mathbb J$ in a suitable metatheory has been explained, we can state our main result.

\begin{theorem}\label{thm:J-Pi11CA}
The following are equivalent over~$\atrs$:
\begin{enumerate}[label=(\roman*)]
\item $\Pi^1_1$-comprehension,
\item the functor~$\mathbb J$ is total, i.\,e., for any dilator~$D$ and any ordinal~$\gamma$ there is an ordinal~$\eta$ such that we have $\mathbb J(D,\gamma)\simeq\eta$.
\end{enumerate}
\end{theorem}

The previous theorem will be derived from Theorem~\ref{thm:theta}. In the latter, we may replace~$\vartheta D$ by a related order~$\psi D$ (see~\cite{freund-rathjen-iterated-Pi11}). Let us note that $\vartheta$ and~$\psi$ refer to two variants of a collapsing function studied in ordinal analysis~\cite{rathjen-weiermann-kruskal}. The order~$\psi D$ has the advantage that it can be characterized by recursion on the dilator~$D$. This will be shown in Section~\ref{sect:Bachmann-rec}, which builds on work of V.\,M.~Abrusci, Girard and J.~van de Wiele~\cite{increasing-goodstein} as well as P.~Uftring~\cite{uftring-inverse-goodstein} (the latter stimulated by a suggestion of A.~Weiermann). Once we have recursive definitions of the ordinals~$\mathbb J(D,\gamma)$ and~$\psi D$, we can use induction to prove comparisons between them (see~Section~\ref{sect:comparison}). In particular, a bound on~$\psi D$ in terms of~$\mathbb J$ will be used to show that~(ii) implies~(i) in Theorem~\ref{thm:J-Pi11CA} (see Corollary~\ref{cor:main-reversal}). The converse implication admits a more direct proof, which shows that admissible sets are closed under~$\mathbb J$ (see Theorem~\ref{thm:Pi11CA-to-J}).

\section{Preliminaries~I: Primitive Recursive Set Theory}\label{sect:prs}

In the present section, we recall basic properties of the primitive recursive set functions and the definition of our base theory~$\atrs$. In particular, we note that primitive recursion suffices to implement recursion over well-founded relations with effectively given predecessors and ranks. The latter will later be used to justify a `guarded' version of recursion over dilators.

The construction of the primitive recursive set functions, which were singled out by Jensen and Karp~\cite{jensen-karp}, starts with the base functions
\begin{align*}
x&\mapsto 0,& (x_0,\ldots,x_n)&\mapsto x_i\text{ for }i\leq n,\\
(x_0,x_1)&\mapsto x_0\cup\{x_1\},& (x_0,x_1,x_2,x_3)&\mapsto\begin{cases}
x_0 & \text{if }x_2\in x_3,\\
x_1 & \text{otherwise}.
\end{cases}
\end{align*}
To obtain all primitive recursive set functions, one closes under composition and recursion over set membership. More precisely, if the function $F,G_i$ and~$H$ are primitive recursive of suitable arity, then so are
\begin{equation*}
\mathbf x\mapsto F(G_0(\mathbf x),\ldots,G_n(\mathbf x))
\end{equation*}
and the unique function~$R$ determined by
\begin{equation}\label{eq:prim-rec}
R(x,\mathbf y)=H\left(\textstyle\bigcup\{R(v,\mathbf y)\,|\,v\in x\},x,\mathbf y\right).
\end{equation}
Note that we use bold letters for tuples. One can derive a version of~(\ref{eq:prim-rec}) that has $R(\cdot,\mathbf y)\restriction x$ as the first argument of~$H$. We say that a class or relation is primitive recursive if it has the form~$\{\mathbf x\,|\,0\in F(\mathbf x)\}$ for some primitive recursive~$F$.  Let us agree that primitive recursive functions may depend on the fixed parameter~$\omega$ (which is useful, e.\,g., to construct transitive closures). Effectively, we thus add a new base function with constant value~$\omega$.

We now recall primitive recursive set theory, as introduced by Rathjen~\cite{rathjen-set-functions}. Its language consists of the binary relation symbols~$\in$ and~$=$ as well as function symbols for all primitive recursive set functions. More precisely, each construction of a primitive recursive set function according to the generating clauses from above corresponds to a function symbol. Each function symbol gives rise to a defining axiom, which reflects the corresponding clause. Primitive recursive set theory with infinity~($\prs$) consists of these defining axioms together with the axioms of equality, extensionality, foundation and infinity. One can derive pairing, union and $\Delta_0$-separation. Let us note that a relation is $\Delta_0$-definable precisely if it is primitive recursive. Using foundation and $\Delta_0$-separation, one can derive $\in$-induction for primitive recursive properties. We refer to Sections~1.1 and~1.2 of~\cite{freund-thesis} for details and further fundamental constructions.

The following principle of induction along an $\in$-ranked predecessor relation will later be used to justify a guarded form of induction over dilators. Its proof is very similar to the proof of an analogous recursion principle from~\cite{freund-thesis}, which we recall below. Here and in the following, we allow implicit parameters, i.\,e., functions may have free variables as suppressed arguments.

\begin{proposition}[$\prs$]\label{prop:pred-ind}
Consider primitive recursive functions~$P$ and~$R$ such that $x\in P(y)$ entails~$R(x)\in R(y)$. If $Q$ is a primitive recursive class with the property that $P(x)\subseteq Q$ entails~$x\in Q$, then $Q$ contains all sets.
\end{proposition}

Before we prove the proposition, we note that $P,R$ and~$Q$ are given as function symbols over which we quantify in the metalanguage. In Proposition~\ref{prop:pred-rec} below, $F$ is represented by a function symbol that is constructed from function symbols for~$P,R,H$ in the metatheory. The uniqueness part of the proposition asserts that the values of~$F$ are unique provably in~$\prs$ (while there are different function symbols that describe~$F$). Subsequent results should be interpreted in a similar~way.

\begin{proof}
For an arbitrary~$x$, we consider $\mathsf{TC}^P(x)=\bigcup_{n\in\omega}\mathsf{TC}_n^P(x)$ with
\begin{equation*}
\mathsf{TC}_0^P(x)=x\quad\text{and}\quad\mathsf{TC}_{n+1}^P(x)=\textstyle\bigcup\{P(y)\,|\,y\in\mathsf{TC}_n^P(x)\},
\end{equation*}
where the latter amounts to a primitive recursion. We can use $\in$-induction on~$r$ to establish the $\Delta_0$-statement
\begin{equation*}
\forall y\in\mathsf{TC}^P(\{x\})\,\big(R(y)=r\rightarrow y\in Q\big).
\end{equation*}
Indeed, if $y\in\mathsf{TC}^P(\{x\})$, then $z\in P(y)$ entails~$z\in\mathsf{TC}^P(\{x\})$ and~$R(z)\in R(y)$, so that we inductively get $P(y)\subseteq Q$ and hence~$y\in Q$. In view of~$x\in\mathsf{TC}^P(\{x\})$, we have thus established~$x\in Q$.
\end{proof}

The following provides a corresponding recursion principle. Let us note that the function $z\mapsto F\restriction z=\{\langle x,F(x)\rangle\,|\,x\in z\}$ is primitive recursive when the same holds for~$F$ (and in particular~$F\restriction z$ exists as a set).

\begin{proposition}[$\prs$; {\cite[Proposition~1.2.10]{freund-thesis}}]\label{prop:pred-rec}
For primitive recursive~$P,R,H$ such that $x\in P(y)$ entails~$R(x)\in R(y)$, there is a unique primitive recursive~$F$ that validates $F(x)=H(F\restriction P(x))$.
\end{proposition}
\begin{proof}
Uniqueness is readily established by the induction principle from the previous proposition. Existence is shown in~\cite{freund-thesis}, where it is required that $R$ maps into the ordinals. This can be arranged by composing with the primitive recursive rank function that is given by $r(x)=\sup\{r(y)+1\,|\,y\in x\}$. The idea is to use induction on the ordinal~$\alpha$ to construct the function
\begin{equation*}
(\alpha,x)\mapsto\big\{\langle y,F(y)\rangle\,\big|\,y\in\mathsf{TC}^P(x)\text{ and }R(y)<\alpha\},
\end{equation*}
where $\mathsf{TC}^P$ is defined as in the previous proof.
\end{proof}

Axiom beta is the statement that any well-founded relation~$r$ admits a function~$f$ that is defined on the field of~$r$ and validates~$f(y)=\{f(x)\,|\,(x,y)\in r\}$ (where $r$ and~$f$ are sets). This is not provable in~$\prs$, though the latter proves the special case where~$r$ is the restriction of~$\in$ to a set, which entails that any model~$(x,\in)$ of extensionality is isomorphic to one where~$x$ is transitive. The axiom of countability asserts that any set~$x$ admits an injection~$g:x\to\omega$. Over~$\prs$, this is equivalent to the statement that~$x$ is empty or admits a surjection~$h:\omega\to x$.

By $\atrs$, we mean the extension of~$\prs$ by axiom beta and the axiom of countability. This theory is conservative over a theory that Simpson~\cite{simpson82} has introduced under the same name, which is formulated in the usual signature~$\{\in,=\}$ without function symbols (see Section~1.4 of~\cite{freund-thesis} for details on conservativity). As shown by Simpson, $\atrs$ is conservative over the theory~$\mathsf{ATR}_0$ in the language of second-order arithmetic, which plays a central role in reverse mathematics. We note that~\cite{simpson82} describes countability as an `optional extra axiom' but includes it in the proof of conservativity. In the presence of countability, $\Pi^1_1$-comprehension is equivalent to the statement that any set is contained in an admissible set, i.\,e., in a transitive model of Kripke-Platek set theory (see, e.\,g., Section~1.4 of~\cite{freund-thesis}, which adopts closely related results from Section~7 of~\cite{jaeger-admissibles}). Modulo the introduction of function symbols, Kripke-Platek set theory is the extension of~$\prs$ by $\in$-induction for arbitrary formulas (which is automatic in transitive models) and the axiom of~$\Delta_0$-collection, which asserts that $\forall x\in a\exists y\,\theta$ entails $\exists b\forall x\in a\exists y\in b\,\theta$ when $\theta$ is a bounded formula.

In the presence of axiom beta, we write $\otp(r)$ for the unique ordinal that is isomorphic to a given well order~$r$. While $\prs$ does not prove that~$\otp$ is a total function on well orders, it can determine order types below a given bound. In particular, the class of well orders with order type below a given bound is primitive recursive, while the class of all well orders is not (as the latter is $\Pi^1_1$-complete). This will later be needed to define the functor~$\mathbb J$ (see the introduction) by a guarded form of recursion over dilators.

\begin{lemma}[$\prs$]\label{lem:otp-restr}
We have a primitive recursive map $(\xi,r)\mapsto\otp_\xi(r)$~with
\begin{equation*}
\otp_\xi(r)=\begin{cases}
\otp(r) & \parbox[t]{.55\textwidth}{if $\xi$ is an ordinal and $r$ is a well order that is isomorphic to an ordinal $\otp(r)<\xi$,}\\[3ex]
\xi & \text{otherwise}.
\end{cases}
\end{equation*}
\end{lemma}
\begin{proof}
We may focus on the case where $\xi$ is an ordinal and~$r$ is a linear order, since it can be singled out by a primitive recursive case distinction. Also, we may assume that the field~$x$ of~$r$ does not contain~$\emptyset$ as an element (otherwise replace it). We have a primitive recursive function~$F$ such that $F(z)$ is the set of all $r$-minimal elements of~$x\backslash z$ (since~$\Delta_0$-separation is primitive recursive). Then $\bigcup F(z)$ is the (necessarily unique) $r$-minimal element of~$x\backslash z$ if such an element exists and equal to~$\emptyset$ otherwise. By recursion over~$\alpha$, we get a primitive recursive~$G$ with
\begin{equation*}
G(\alpha)=\bigcup F(G"(\alpha))\quad\text{for}\quad G"(\alpha)=\{G(\beta)\,|\,\beta<\alpha\}.
\end{equation*}
We show that the equation from the proposition is satisfied if we adopt the primitive recursive definition
\begin{equation*}
\otp_\xi(r)=\min\big\{\alpha\leq\xi\,\big|\,G"(\alpha)=x\text{ or }\alpha=\xi\big\}.
\end{equation*}
First assume $r$ is isomorphic to an ordinal~$\otp(r)<\xi$. We write $c:\otp(r)\to x$ for the given isomorphism. An induction on~$\gamma<\otp(r)$ shows $c(\gamma)=G(\gamma)$. So~$\otp(r)$ is indeed the minimal~$\alpha<\xi$ with $G"(\alpha)=x$. Now assume $r$ is not isomorphic to any ordinal~$\alpha<\xi$. Towards a contradiction, we assume $\alpha<\xi$ validates $G"(\alpha)=x$. In view of~$\emptyset\notin x$, the value $G(\gamma)$ must be $r$-minimal in~$x\backslash G"(\gamma)$ for any~$\gamma<\alpha$. Thus $G\restriction\alpha$ is order preserving and hence an isomorphism between~$\alpha$ and $r$.
\end{proof}

\section{Preliminaries~II: Induction on Dilators}\label{sect:prelim-dil}

In this section, we first recall the definition and basic properties of dilators. We then discuss induction and recursion over dilators and show that `guarded' versions of these principles are available in primitive recursive set theory. Except for these guarded versions, all material in the section goes back to Girard~\cite{girard-pi2}, though the presentation is closer to~\cite{freund-computable}.

We write $\lo$ for the category of linear orders and embeddings. By $\nat$ we denote the full subcategory with objects~$n=\{0,\ldots,n-1\}$ for~$n\in\mathbb N$ (with $0<\ldots<n-1$). Morphisms of~$\lo$ are compared pointwise, i.\,e., for order embeddings~$f,g:X\to Y$ we write $f\leq g$ if we have $f(x)\leq_Y g(x)$ for all~$x\in X$. A functor~$D$ between $\nat$ and~$\lo$ is called monotone if $f\leq g$ entails~$D(f)\leq D(g)$.

The finite powerset functor on the category of sets is given by
\begin{align*}
[X]^{<\omega}&=\{a\subseteq X\,|\,a\text{ finite}\},\\
[f]^{<\omega}(a)&=\{f(x)\,|\,x\in a\}\in[Y]^{<\omega}\text{ for $f:X\to Y$ and $a\in[X]^{<\omega}$.}
\end{align*}
We omit the forgetful functor from orders to sets and often consider subsets of an order as suborders. For the following definition, this means that $\supp$ is a natural transformation between functors from~$\lo$ to sets. In general, we write $\eta:D\Rightarrow E$ to express that~$\eta$ is a natural transformation between functors~$D$ and~$E$, which consists of morphisms~$\eta_X:D(X)\to E(X)$ for all objects~$X$ in the relevant category. Let us also agree to write~$\rng(f)=\{f(x)\,|\,x\in X\}$ for the range of a function~$f$.

We would like to define predilators as certain endofunctors on~$\lo$. As the latter is class-sized, we officially work with their restrictions to~$\nat$. The extension to~$\lo$ will later be achieved by taking direct limits. With this extension in mind, we use~$[\cdot]^{<\omega}$ even in the context of finite sets, where it coincides with the full powerset.

\begin{definition}\label{def:predil}
A predilator consists of a monotone functor~$D:\nat\to\lo$ and a natural transformation~$\supp:D\Rightarrow[\cdot]^{<\omega}$ such that the `support condition'
\begin{equation*}
\rng(f)\subseteq\supp_n(\sigma)\quad\Rightarrow\quad\sigma\in\rng(D(f))
\end{equation*}
holds for any morphism~$f:m\to n$ of~$\nat$ and any $\sigma\in D(n)$.
\end{definition}

Note that the converse of the support condition is automatic, as naturality yields
\begin{equation*}
\supp_n(D(f)(\sigma_0))=[f]^{<\omega}(\supp_m(\sigma_0))\subseteq\rng(f).
\end{equation*}
We write $|a|$ for the cardinality of a finite set~$a$. When the latter is a subset of an order~$X$, the unique embedding~$|a|\to X$ with range~$a$ will be denoted by~$\en_a^X$. We observe that $\supp_n(\sigma)$ is determined as the smallest~$a\subseteq n$ with $\sigma\in\rng(D(\en_a^n))$. The existence of a smallest (not just minimal)~$a$ with this property is equivalent to the condition that~$D$ preserves pullbacks. When the latter is the case, the induced functions~$\supp_n:D(n)\to[n]^{<\omega}$ satisfy the support condition and are automatically natural (essentially by Girard's normal form theorem~\cite{girard-pi2}). In other words, the natural transformation~$\supp$ in the previous definition is not an extra piece of data but just an explicit witness that $D$ preserves pullbacks. The aforementioned extension to~$\lo$ will preserve direct limits by construction. This means that our definition of predilators is equivalent to the one of Girard.

Any element $\sigma\in D(n)$ has a `normal form' $\sigma=D(\en_a^n)(\sigma_0)$ with~$a=\supp_n(\sigma)$ and~$\sigma_0\in D(|a|)$, as mentioned above. By naturality, we have
\begin{equation*}
[\en_a^n]^{<\omega}(\supp_{|a|}(\sigma_0))=\supp_n(\sigma)=a,
\end{equation*}
which entails~$\supp_{|a|}(\sigma_0)=|a|$. This motivates the following notion.

\begin{definition}
The trace of a predilator~$D$ is defined as
\begin{equation*}
\tr(D)=\{(\sigma,n)\,|\,n\in\mathbb N\text{ and }\sigma\in D(n)\text{ with }\supp_n(\sigma)=n\}.
\end{equation*}
\end{definition}

Next, we want to extend a given predilator~$D$ into a functor~$\overline D:\lo\to\lo$. In view of the aforementioned normal forms, the idea is to define~$\overline D(X)$ as a set of formal expressions~$\overline D(\en_a^X)(\sigma)$ with $(|a|,\sigma)\in\tr(D)$. The order between two expressions with data~$(\sigma,a)$ and $(\tau,b)$ can be determined in~$D(|a\cup b|)$. To make this explicit, we first agree to write $\en_a$ at the place of $\en_a^a:|a|\to a$. Each embedding~$f:a\to b$ of finite orders determines a morphism~$|f|:|a|\to|b|$ of~$\nat$, which is characterized by
\begin{equation*}
\en_b\circ|f|=f\circ\en_a.
\end{equation*}
When~$a$ is a suborder of~$X$, we write $\iota_a:a\hookrightarrow X$ for the inclusion.

\begin{definition}\label{def:dil-extend}
Consider a predilator~$D$ with support functions~$\supp$. For each linear order~$X$, we declare
\begin{gather*}
\overline D(X)=\{(\sigma,a)\,|\,a\in[X]^{<\omega}\text{ and }(\sigma,|a|)\in\tr(D)\},\\
(\sigma,a)<(\tau,b)\text{ in }\overline D(X)\quad\Leftrightarrow\quad D(|\iota_a^{a\cup b}|)(\sigma)<D(|\iota_b^{a\cup b}|)(\tau)\text{ in }D(|a\cup b|).
\end{gather*}
Furthermore, we associate functions
\begin{alignat*}{3}
\overline D(f):\overline D(X)&\to\overline D(Y)\quad&&\text{with}\quad\overline D(f)(\sigma,a)=(\sigma,[f]^{<\omega}(a)),\\
\overline\supp_X:\overline D(X)&\to[X]^{<\omega}\quad&&\text{with}\quad\overline\supp_X(\sigma,a)=a,
\end{alignat*}
where $f:X\to Y$ is assumed to be an order embedding.
\end{definition}

One can verify that~$\overline D$ satisfies the conditions from Definition~\ref{def:predil} with~$\lo$ at the place of~$\nat$. Conversely, any functor~$E:\lo\to\lo$ that satisfies these conditions is naturally isomorphic to~$\overline D$ when~$D$ is the restriction of~$E$ to~$\nat$ (at least when the class-sized~$E$ is suitably definable). Here the isomorphism maps~$(\sigma,a)\in\overline D(X)$ to~$E(\en_a^X)(\sigma)$. In view of this close correspondence, we switch freely between~$D$ and~$\overline D$, e.\,g., by referring to~$(\sigma,a)$ as an element of~$D(X)$. This element will also be written as $(\sigma;a_0,\ldots,a_{n-1})$ or as $(\sigma;a_0,\ldots,a_{n-1};X)$ for $a=\{a_0,\ldots,a_{n-1}\}$, where we always assume $a_0<_X\ldots<_X a_{n-1}$. The latter coincides with Girard's notation.

\begin{definition}\label{def:dil}
A predilator~$D$ is called a dilator when~$\overline D(X)$ is well-founded for every well order~$X$.
\end{definition}

Let us note that everything so far remains meaningful if we take Definition~\ref{def:predil} without the condition that $D$ must be monotone. The latter is in fact automatic when $D$ is a dilator (see Proposition~2.3.10 of~\cite{girard-pi2}). It is sometimes important to have the condition for predilators as well.

The material that we have seen so far can also be found in Section~2 of~\cite{freund-computable}, where it is developed in primitive recursive set theory~($\prs$). In the following, we work towards the principles of induction and recursion over dilators. These are due to Girard~\cite{girard-pi2} but have apparently not been discussed in the context of~$\prs$ yet.

\begin{definition}\label{def:connected}
Consider a predilator~$D$. To define a binary relation~$\ll$ on~$\tr(D)$, we stipulate that $(\sigma_0,n_0)\ll(\tau_1,n_1)$ holds if we have $D(f_0)(\sigma_0)<D(f_1)(\sigma_1)$ for all embeddings $f_i:n_i\to n_0+n_1$. We declare that $\sigma\equiv\tau$ holds if we have neither $\sigma\ll\tau$ nor~$\tau\ll\sigma$. A predilator~$D$ is called connected if we have $D\neq 0$ and $\sigma\equiv\tau$ holds for all~$\sigma,\tau\in\tr(D)$.
\end{definition}

Note that we write~$0$ for both the empty order and the dilator~$D$ with $D(X)=0$ for every order~$X$. Similarly, we will write~$1$ for the dilator~$D$ with~$D(X)=1$. In the definition of~$\ll$, we can equivalently replace~$n_0+n_1$ by any larger~$N$ (factor by an embedding~$\iota:n_0+n_1\to N$ with $\rng(f_0)\cup\rng(f_1)\subseteq\rng(\iota)$). It follows that~$\ll$ is transitive and indeed a partial order.

\begin{lemma}[$\prs$]
If we have $\sigma_0\ll\sigma_1\equiv\sigma_2$ or $\sigma_0\equiv\sigma_1\ll\sigma_2$, we get $\sigma_0\ll\sigma_2$. Furthermore, $\equiv$ is an equivalence relation.
\end{lemma}
\begin{proof}
It is straightforward to see that $\equiv$ is reflexive and symmetric. Transitivity follows once the first claim of the lemma is established. Indeed, given $\sigma_0\equiv\sigma_1\equiv\sigma_2$, we cannot have $\sigma_0\ll\sigma_2$, since the latter and $\sigma_2\equiv\sigma_1$ would yield~$\sigma_0\ll\sigma_1$. Similarly, we must have $\sigma_2\not\ll\sigma_0$, as needed to get~$\sigma_0\equiv\sigma_2$.

Let us now show that $\sigma_0\ll\sigma_2$ follows from $\sigma_0\ll\sigma_1\equiv\sigma_2$ (the argument for $\sigma_0\equiv\sigma_1\ll\sigma_2$ being similar). Write~$\sigma_i=(\tau_i,n_i)$ and consider arbitrary embeddings $f_0:n_0\to N$ and $f_2:n_2\to N$ for some large~$N$. Given $\sigma_2\not\ll\sigma_1$, we find embeddings~$f_2':n_2\to N$ and $f_1:n_1\to N$ with $D(f_2')(\tau_2)\geq D(f_1)(\tau_1)$. Increasing~$N$ if necessary, we may assume that $f_1$ and~$f_2'$ have range entirely below~$f_0$ and $f_2$, so that we have~$f_2'\leq f_2$. We then get $D(f_2')(\tau_2)\leq D(f_2)(\tau_2)$ since predilators are monotone. Due to $\sigma_0\ll\sigma_1$, we also have $D(f_0)(\tau_0)<D(f_1)(\tau_1)$. Together, the previous inequalities yield $D(f_0)(\tau_0)<D(f_2)(\tau_2)$, as needed for~$\sigma_0\ll\sigma_2$.
\end{proof}

For predilators~$D_\alpha$ indexed by a linear order~$L$, we get a predilator~$\sum_{\alpha\in L}D_\alpha$~with
\begin{gather*}
\textstyle\sum_{\alpha\in L}D_\alpha(X)=\{(\alpha,\sigma)\,|\,\alpha\in L\text{ and }\sigma\in D_\alpha(X)\},\\
(\alpha,\sigma)<(\beta,\tau)\quad\Leftrightarrow\quad\alpha<_L\beta\text{ or }(\alpha=\beta\text{ and }\sigma<\tau\text{ in }D_\alpha(X)).
\end{gather*}
Let us note that only orders~$X$ of the form~$\{0,\ldots,n-1\}$ are considered for the official construction with predilators on~$\nat$. For an order embedding~$f:X\to Y$ and an element~$(\alpha,\sigma)$ of $\sum_{\alpha\in L}D_\alpha(X)$, we define
\begin{equation*}
\textstyle\sum_{\alpha\in L}D_\alpha(f)(\alpha,\sigma)=(\alpha,D_\alpha(f)(\sigma)).
\end{equation*}
Writing $\supp^\alpha_X:D_\alpha(X)\to[X]^{<\omega}$ for the support functions associated with~$D_\alpha$, we put $\supp_X(\alpha,\sigma)=\supp^\alpha_X(\sigma)$ for~$(\alpha,\sigma)\in\sum_{\alpha\in L}D_\alpha(X)$. This clearly validates the conditions from Definition~\ref{def:predil}. The element~$(\alpha,\sigma)$ of~$\sum_{\alpha\in L}D_\alpha(X)$ may also be written as~$\sum_{\gamma<\alpha}D_\alpha(X)+\sigma$. If we have $D_\alpha\neq 0$ for all~$\alpha\in L$, then $\sum_{\alpha\in L}D_\alpha$ is a dilator precisely when~$L$ is a well order and every~$D_\alpha$ is a dilator.

For predilators~$D$ and $E$, we write~$D\leq E$ if a natural transformation~$\eta:D\Rightarrow E$ exists. Any such~$\eta$ respects the associated supports, i.\,e., we have~$\supp^D={\supp^E}\circ\eta$ (see Proposition~2.3.15 of~\cite{girard-pi2} or Lemma~2.19 of~\cite{freund-rathjen_derivatives}). This means that the naturality squares for~$\eta$ are pullbacks, i.\,e., that $\eta$ is Cartesian. Crucially, it follows that
\begin{equation*}
\tr(\eta):\tr(D)\to\tr(E)\quad\text{with}\quad\tr(\eta)(n,\sigma)=(n,\eta_n(\sigma))
\end{equation*}
is well-defined. In terms of the notation that was discussed before Definition~\ref{def:dil}, the embedding~$\eta_X:D(X)\to E(X)$ can be written as
\begin{equation*}
D(X)\ni(\sigma,a)\mapsto(\eta_n(\sigma),a)\in E(X).
\end{equation*}
Using that~$\eta$ is natural, one obtains
\begin{equation*}
\sigma\ll\tau\text{ in }\tr(D)\quad\Leftrightarrow\quad\tr(\eta)(\sigma)\ll\tr(\eta)(\tau)\text{ in }\tr(E).
\end{equation*}
Any $A\subseteq\tr(D)$ gives rise to a predilator~$D[A]$ with
\begin{align*}
D[A](X)&=\big\{\sigma\in D(X)\,\big|\,(\sigma_0,|a|)\in A\text{ for }\sigma=D(\en_a^X)(\sigma_0)\text{ with }a=\supp_X(\sigma)\big\}\\
{}&=\{(\sigma;a_0,\ldots,a_{n-1};X)\,|\,(\sigma,n)\in A\}.
\end{align*}
The inclusion maps from~$D[A](X)$ into~$D(X)$ form a natural transformation. Conversely, any natural transformation~$\eta:D\Rightarrow E$ factors through~$D\cong E[A]$ for a unique~$A\subseteq\tr(E)$ (see Theorem~4.2.5 of~\cite{girard-pi2} or more explicitly Section~2 of~\cite{freund-pilot}).

\begin{proposition}[$\prs$]\label{prop:sums}
(a) There is a primitive recursive set function that computes, for a given predilator~$D$, a linear order~$L$ and connected predilators~$D_\alpha$ such that we have $D\cong\sum_{\alpha\in L}D_\alpha$.

(b) Given a natural transformation~$\eta:\sum_{\alpha\in K}D_\alpha\Rightarrow\sum_{\beta\in L}E_\beta$ for connected~$D_\alpha$ and~$E_\beta$, we get an order embedding~$f:K\to L$ such that we have $D_\alpha\leq E_{f(\alpha)}$ for all~$\alpha\in K$. If $\eta$ is an isomorphism, then so is~$f$ and we have~$D_\alpha\cong E_{f(\alpha)}$.
\end{proposition}
\begin{proof}
(a) In view of the previous lemma, we may define~$L$ as the quotient of the partial order $(\tr(D),\ll)$ by the compatible equivalence relation~$\equiv$. Using primitive recursion, one can compute~$L$ as the image of the quotient map~$\tr(D)\ni a\to[a]$, which can itself be computed from~$D$. For~$\alpha\in L$ we set $D_\alpha=D[\alpha]$, which is explained in the paragraph before the proposition, as we have~$\alpha\subseteq\tr(D)$. Let us note that the~$D_\alpha$ are connected by construction. In view of $D[\alpha](X)\subseteq D(X)$, we declare that the desired isomorphism maps~$(\gamma,\sigma)\in\sum_{\alpha\in L}D_\alpha(X)$ to~$\sigma\in D(X)$. This map is surjective because of~$\tr(D)=\bigcup L$. To see that it is an embedding, consider~$\sigma_i\in D[\gamma_i](X)$ with~$\gamma_0<_L\gamma_1$. Let $\sigma_i=D(\en_{a(i)}^X)(\tau_i)$ with $a(i)=\supp_X(\sigma_i)$ and~$(|a(i)|,\tau_i)\in\gamma_i$. We have $(|a(0)|,\tau_0)\ll(|a(1)|,\tau_1)$, so that the definition of~$\ll$ immediately yields~$\sigma_0<_{D(X)}\sigma_1$.

(b) One checks that $K$ is isomorphic to the quotient of~$(\tr(\sum_{\alpha\in K}D_\alpha),\ll)$ by~$\equiv$. Thus the aforementioned map~$\tr(\eta):\tr(\sum_{\alpha\in K}D_\alpha)\to\tr(\sum_{\beta\in L}E_\beta)$ induces the desired embedding~$f:K\to L$. We may write
\begin{equation*}
\eta_X(\alpha,\sigma)=(f(\alpha),\eta^\alpha_X(\sigma))
\end{equation*}
for embeddings~$\eta^\alpha_X:D_\alpha(X)\to E_{f(\alpha)}(X)$, which are natural and witness that we have $D_\alpha\leq E_{f(\alpha)}$. When~$\eta$ is an isomorphism, its inverse gives rise to inverses of~$f$ and the~$\eta^\alpha_X$, so that these are isomorphisms.
\end{proof}

Following Girard~\cite{girard-pi2}, we now classify predilators into four types. Note that a (not necessarily well-founded) linear order~$L$ is called a successor if it has a biggest element~$\gamma$, in which case we write~$L=\gamma+1$. A non-empty order that is not a successor is called a limit.

\begin{definition}\label{def:dil-types}
A predilator~$D\cong\sum_{\alpha\in L}D_\alpha$ with connected~$D_\alpha$ has type~$0$ if~$L=0$ (and hence~$D=0$), type~$1$ if~$L=\gamma+1$ is a successor and~$D_\gamma=1$, type~$\omega$ if~$L$ is a limit and type~$\Omega$ if $L=\gamma+1$ is a successor and~$D_\gamma\neq 1$.
\end{definition}

In the sequel, the assumption that the~$D_\alpha$ in~$D\cong\sum_{\alpha\in L}D_\alpha$ are connected will often be left implicit. When we have~$L=\gamma+1$, we often write~\mbox{$D=E_0+E_1$} with the implicit assumption that~$E_1$ is connected, i.\,e., that we have~$E_0=\sum_{\alpha<\gamma}D_\alpha$ and~$E_1=D_\gamma$. To formulate principles of induction and recursion, we need to identify the predecessors of a dilator. This is most challenging in the case of type~$\Omega$. Let us first note that we must have~$n>0$ for any~$(n,\sigma)\in\tr(D)$ when~$D\neq 1$ is connected. Indeed, it is straightforward to show that a trace element of the form~$(0,\sigma)$ is $\ll$-comparable to any other. The following is well-defined by a result of Girard~\cite{girard-pi2} (see Proposition~2.7 of~\cite{freund-zoo} for a statement over a weak base theory).

\begin{definition}\label{def:import-coeff}
Consider a connected predilator~$D\neq 1$. For $(n,\sigma)\in\tr(D)$, we define~$i(\sigma)$ as the unique~$i<n$ such that
\begin{equation*}
f(i)<g(i)\quad\Rightarrow\quad D(f)(\sigma)<D(g)(\sigma)
\end{equation*}
holds for all embeddings~$f,g:n\to 2n$.
\end{definition}

As in the case of~$\ll$, the definition of~$i(\sigma)$ does not change when we replace~$2n$ by a larger number. Using the notation from the paragraph before Definition~\ref{def:dil}, we get the following, where $(\sigma,n)$ and~$(\tau,m)$ do not need to be the same trace element (see Section~3.2 of~\cite{girard-pi2} or Theorem~2.11 of~\cite{freund-zoo}).

\begin{lemma}[$\prs$]\label{lem:import-coeff}
For a connected predilator~$D$ and any order~$X$, we have
\begin{equation*}
a_{i(\sigma)}<_X b_{i(\tau)}\quad\Rightarrow\quad(\sigma;a_0,\ldots,a_m;X)<_{D(X)}(\tau;b_0,\ldots,b_n;X).
\end{equation*}
\end{lemma}

We write~$X+Y$ for the sum of linear orders and refer to its elements by~$x$ and~$X+y$ for $x\in X$ and~$y\in Y$ (note that we always have~$x<X+y$). Let $x<X$ express that~$x$ lies in the left summand. For an embedding~$f:Y\to Y'$, we define $X+f:X+Y\to X+Y'$ by $(X+f)(x)=x$ and~$(X+f)(X+y)=X+f(y)$.

\begin{definition}\label{def:sep-vars}
Given a connected predilator~$D\neq 1$ and an ordinal~$\gamma$, we declare
\begin{equation*}
\{D\}^\gamma(X)=\big\{(\sigma;a_0,\ldots,a_n)\in D(\gamma+X)\,|\,i(\sigma)=\max\{i\leq n\,|\,a_i<\gamma\}\big\}.
\end{equation*}
For an embedding~$f:X\to Y$, we define $\{D\}^\gamma(f):\{D\}^\gamma(X)\to\{D\}^\gamma(Y)$ as the restriction of~$D(\gamma+f)$ to the indicated (co-)domain. To turn~$\{D\}^\gamma$ into a predilator, we declare that $(\sigma;\gamma_0,\ldots,\gamma_{i(\sigma)},\gamma+x_0,\ldots,\gamma+x_{m-1})$ has support~$\{x_0,\ldots,x_{m-1}\}$. When~$D+E$ has type~$\Omega$ (with connected~$E\neq 1$), we set~$\{D+E\}^\gamma=D+\{E\}^\gamma$.
\end{definition}

Let us note that the given construction coincides with the separation of variables due to Girard, who writes~$\mathsf{SEP}(D)(X,\gamma)$ at the place of~$\{D\}^\gamma(X)$. The latter~notation is used by P\"appinghaus~\cite{paeppinghaus-bachmann}.

Our next result (which is implicit in the proof of Theorem~3.5.1 from~\cite{girard-pi2}) shows how~$\{D\}^\gamma$ can be seen as a predecessor of~$D$.

\begin{lemma}[$\prs$]\label{lem:sep-pred}
Consider ordinals~$\gamma$ and $\delta>0$ with $\gamma<\omega^\delta$. For any predilator~$D$ of type~$\Omega$, we can embed~$\{D\}^\gamma(\omega^\delta)$ into a proper initial segment of~$D(\omega^\delta)$. 
\end{lemma}
\begin{proof}
It is straightforward to reduce to the case where $D\neq 1$ is connected. We have $\gamma+\omega^\delta=\omega^\delta$ and hence~$\{D\}^\gamma(\omega^\delta)\subseteq D(\omega^\delta)$. Choose $(\sigma;a_0,\ldots,a_n)\in D(\omega^\delta)$ with $a_{i(\sigma)}=\gamma$. This is possible due to~$\gamma+\omega\leq\omega^\delta$. Lemma~\ref{lem:import-coeff} entails that the chosen element bounds $\{D\}^\gamma(\omega^\delta)$.
\end{proof}

In the introduction, we have described a characterization of countable dilators based on their value on the least uncountable cardinal~$\Omega$. We now see that this characterization coincides with the one from Definition~\ref{def:dil-types}.

\begin{definition}\label{def:dil-preds}
For a predilator~$D$ and an ordinal~$\xi$, we define
\begin{equation*}
P_\xi(D)=\begin{cases}
\emptyset & \text{if $D$ has type~$0$},\\
\{D'\} & \text{if $D\cong D'+1$ has type~$1$},\\
\big\{\textstyle\sum_{\alpha<\nu}D_\alpha\,\big|\,\nu\in L\big\} & \text{if $D=\sum_{\alpha\in L}D_\alpha$ has type~$\omega$},\\
\big\{\{D\}^\gamma\,\big|\,\gamma<\xi\big\} & \text{if $D$ has type~$\Omega$}.
\end{cases}
\end{equation*}
In the third case, it is understood that the~$D_\alpha$ are connected, and $\alpha<\nu$ means that the sum is indexed by the suborder~$\{\alpha\in L\,|\,\alpha<_L\nu\}$.
\end{definition}

Let us note that $D'$ is a predecessor of~$D$ in the sense of Girard (see Theorem~3.5.1 of~\cite{girard-pi2}) precisely when we have $D'\in P_\xi(D)$ for some ordinal~$\xi$. The following result will be needed later.

\begin{lemma}[$\prs$]\label{lem:Omega-nat-trans}
(a) For~$D$ of type~$\Omega$ and~$\gamma\leq\delta$, we have~$\{D\}^\gamma\leq\{D\}^\delta$.

(b) Given $D\leq E$ for~$E$ of type~$\Omega$, we get $D\leq\{E\}^0$ or the predilator $D$ is also of type~$\Omega$ and we have~$\{D\}^\gamma\leq\{E\}^\gamma$ for every ordinal~$\gamma$.
\end{lemma}
\begin{proof}
Part~(a) is a straightforward consequence of the definition. To establish~(b), we first assume that~$E\neq 1$ is connected. Then $D$ must be connected in view of Proposition~\ref{prop:sums}(b). We must also have~$D\neq 1$, since~$1\leq E$ would require a trace element~$(0,\sigma)\in\tr(E)$ (cf.~the discussion before Definition~\ref{def:import-coeff}). Consider a natural transformation~$\eta:D\Rightarrow E$ and recall that the embedding $\eta_{\gamma+X}$ is given by
\begin{equation*}
D(\gamma+X)\ni(\sigma;a_0,\ldots,a_n)\mapsto(\eta_n(\sigma);a_0,\ldots,a_n)\in E(\gamma+X).
\end{equation*}
To see that this map restricts to an embedding of~$\{D\}^\gamma(X)$ into~$\{E\}^\gamma(X)$, one should note that~$i(\sigma)=i(\eta_n(\sigma))$ follows from the naturality of~$\eta$. For the general case, consider~$E=E_0+E_1$ with connected~$E_1\neq 1$. Using Proposition~\ref{prop:sums}(b) again, we get $D\leq E_0=\{E\}^0$ or $D=D_0+D_1$ with~$D_i\leq E_i$. In the latter case, the previous considerations yield $\{D_1\}^\gamma\leq\{E_1\}^\gamma$, which readily implies the claim.
\end{proof}

The previous material in this section is due to Girard~\cite{girard-pi2}. We now present induction and recursion principles for dilators that are new in this form. Specifically, Girard has proved more general versions of these principles in a strong metatheory (see again Theorem~3.5.1 of~\cite{girard-pi2}). Our contribution is the identification of restricted versions that are available already in~$\atrs$.

\begin{theorem}[$\atrs$; Guarded Induction on Dilators]\label{thm:guarded-ind}
Consider a primitive recursive class~$Q$ and function~$G$ such that $G(D)$ is an ordinal for every dilator~$D$. Assume that a dilator~$D$ lies in~$Q$ whenever we have $P_{G(D)}(D)\subseteq Q$. Then~$Q$ contains every dilator.
\end{theorem}
\begin{proof}
Under the assumptions of the theorem, we show~$D\in Q$ for a fixed dilator~$D$. To find a bound~$\xi$ on all relevant values of~$G$, we set $P^\omega(D)=\bigcup_{n<\omega}P^n(D)$ with
\begin{equation*}
P^0(D)=\{D\}\quad\text{and}\quad P^{n+1}(D)=\bigcup\{P_{G(E)}(E)\,|\,E\in P^n(D)\}.
\end{equation*}
Even in the absence of $\Pi^1_2$-induction, we can show that every~$E\in P^n(D)$ is a dilator, as another recursion on~$n$ yields embeddings~$E(X)\to D(\gamma+X)$ for suitable ordinals~$\gamma$. Let us now put
\begin{equation*}
\xi=\sup\{G(E)\,|\,E\in P^\omega(D)\}.
\end{equation*}
Using axiom beta, we may pick an ordinal~$\eta$ that exceeds the order type of~$D(\omega^{1+\xi})$. Due to Lemma~\ref{lem:sep-pred} (see also Lemma~\ref{lem:otp-restr}), an induction on~$n$ yields
\begin{equation*}
\otp_\eta(E(\omega^{1+\xi}))<\eta\quad\text{for every }E\in P^n(D).
\end{equation*}
As part of the induction step, one can observe that we even have
\begin{equation*}
\otp_\eta(E'(\omega^{1+\xi}))<\otp_\eta(E(\omega^{1+\xi}))\quad\text{for }E'\in P_{G(E)}(E)\text{ with }E\in P^\omega(D).
\end{equation*}
We can thus conclude by Proposition~\ref{prop:pred-ind} with predecessors
\begin{equation*}
P(E)=\begin{cases}
P_{G(E)}(E) & \text{for }E\in P^\omega(D),\\
\emptyset & \text{otherwise},
\end{cases}
\end{equation*}
ranks $R(E)=\otp_\eta(E(\omega^{1+\xi}))$ and~$\{E\,|\,E\in P^\omega(D)\to E\in Q\}$ at the place of~$Q$.
\end{proof}

Note that the proof of the previous theorem reduces to the case where~$G(D)=\xi$ is constant. In fact, it is not quite clear what other~$G$ would be useful in applications. As the collapse of an arbitrary well order onto an ordinal is not primitive recursive, we cannot even define~$G(D)$ to be a value like $D(\omega)$ (note that Definition~\ref{def:dil-extend} only produces well orders and not ordinals). For this reason, we formulate the following for constant~$G$ only.

\begin{theorem}[$\atrs$; Guarded Recursion on Dilators]\label{thm:guarded-rec}
Given a primitive recursive function~$H$, we get a primitive recursive~$F$ such that
\begin{equation*}
F(D,\eta,\xi)=F(D,\eta,\xi')=H\big(F(\cdot,\eta,\xi)\restriction P_\eta(D)\big)
\end{equation*}
holds for any dilator~$D$ and any ordinals~$\eta$ and~$\xi\leq\xi'$ such that $D(\omega^{1+\eta})$ has order type below $\xi$.
\end{theorem}
Note that implicit parameters are always permitted, as agreed in the paragraph before Proposition~\ref{prop:pred-ind}. In particular, $H$ may refer to~$D$ as an additional argument. It should not refer to~$\xi$ if we want $F(D,\eta,\xi)=F(D,\eta,\xi')$ as indicated. In applications, we also want $F(D,\eta,\xi)$ to be the same for all large enough~$\eta$, but this seems hard to guarantee on a general level.
\begin{proof}
For a predilator~$D$ and ordinals~$\eta,\xi$, we put $R(D)=\otp_\xi(D(\omega^{1+\eta}))$ and
\begin{equation*}
P(D)=\begin{cases}
P_\eta(D) & \text{if }\otp_\xi(D(\omega^{1+\eta}))<\xi,\\
\emptyset & \text{otherwise}.
\end{cases}
\end{equation*}
Lemma~\ref{lem:sep-pred} ensures that $D'\in P(D)$ implies $R(D')<R(D)$, as in the previous proof. By Proposition~\ref{prop:pred-rec}, we get a primitive recursive~$F$ with
\begin{equation*}
F(D,\eta,\xi)=H(\{(E,F(E,\eta,\xi))\,|\,E\in P(D)\}).
\end{equation*}
When $D(\omega^{1+\eta})$ is well-founded with order type below~$\xi$, we have $P(D)=P_\eta(D)$, which yields the second equation from the lemma. For fixed~$\eta$ and $\xi\leq\xi'$, we get
\begin{equation*}
\otp_\xi(D(\omega^{1+\eta}))<\xi\quad\Rightarrow\quad F(D,\eta,\xi)=F(D,\eta,\xi')
\end{equation*}
by guarded induction on~$D$, i.\,e., by the previous proposition (with $G(D)=\eta$).
\end{proof}

Note that we could have proved the theorem without axiom beta and the assumption that~$D$ is a dilator, since an ordinal bound on~$D(\omega^{1+\eta})$ was given explicitly. However, we then typically need axiom beta to find a bound that allows us to apply the theorem. 

\section{The Functor~$\mathbb J$ of P\"appinghaus}\label{sect:J}

In this section, we show that our base theory~$\atrs$ can represent the functor~$\mathbb J$ of P\"appinghaus (see~\cite{paeppinghaus-bachmann} with $g=1$) as a partial object. We then show that the totality of~$\mathbb J$ follows from $\Pi^1_1$-comprehension (used in the form of admissible sets).

The following definition mimics the clauses for~$\mathbb J$ that were given in the introduction (see also Proposition~\ref{prop:clauses-J} below). It relies on Theorem~\ref{thm:guarded-rec}, which is used to construct the characteristic function of the desired relation. Note that clause~(iv) below refers to~$\{D\}^\alpha$ only when we have~$\alpha\leq\delta<\eta$ and hence~$\{D\}^\alpha\in P_\eta(D)$. We point out that the condition $\otp(D(\omega^{1+\eta}))<\xi$ is inherited from the cited theorem.

\begin{definition}\label{def:J}
Let $\mathbb J(D,\gamma)\simeq_\xi^\eta\delta$ be a primitive recursive relation that holds of a dilator~$D$ and ordinals~$\gamma,\delta,\eta,\xi$ precisely when $\delta<\eta$ and $\otp(D(\omega^{1+\eta}))<\xi$ hold and one of the following applies:
\begin{enumerate}[label=(\roman*)]
\item $D$ has type~$0$ and we have $\delta=\gamma$,
\item $D\cong D'+1$ has type~$1$ and~$\delta=\delta'+1$ is a successor with~$\mathbb J(D',\gamma)\simeq_\xi^\eta\delta'$,
\item $D=\sum_{\alpha<\lambda}D_\alpha$ has type~$\omega$ (with connected~$D_\alpha$) and $\delta=\sup_{\nu<\lambda}f(\nu)$ holds for a function~$f:\lambda\to\delta+1$ with $\mathbb J(\sum_{\alpha<\nu}D_\alpha,\gamma)\simeq_\xi^\eta f(\nu)$ for all~$\nu<\lambda$,
\item $D$ has type~$\Omega$ and we have $\delta=\alpha+\beta$ for ordinals~$\alpha,\beta$ such that we have $\mathbb J(\{D\}^0,\gamma)\simeq_\xi^\eta\alpha$ and $\mathbb J(\{D\}^\alpha,\gamma)\simeq_\xi^\eta\beta$.
\end{enumerate}
We declare that $\mathbb J(D,\gamma)\simeq\delta$ holds if we have $\mathbb J(D,\gamma)\simeq_\xi^\eta\delta$ for some ordinals~$\eta,\xi$. 
\end{definition}

Concerning clause~(iv), it is instructive to observe that we have~$\{D_0+E\}^0=D_0$ for connected~$E\neq 1$ (cf.~Definition~\ref{def:sep-vars}). Let us show that our relation is largely independent of~$\eta$ and~$\xi$.

\begin{lemma}[$\atrs$]\label{lem:J-indep}
We have
\begin{equation*}
\mathbb J(D,\gamma)\simeq_\xi^\eta\delta\quad\Leftrightarrow\quad\mathbb J(D,\gamma)\simeq_\zeta^\mu\delta
\end{equation*}
whenever we have~$\delta<\eta,\mu$ as well as $\otp(D(\omega^{1+\eta}))<\xi$ and $\otp(D(\omega^{1+\mu}))<\zeta$.
\end{lemma}
\begin{proof}
For fixed~$\gamma,\eta,\xi,\mu,\zeta$, guarded induction on~$D$ (Theorem~\ref{thm:guarded-ind} with $G(D)=\eta$) shows that the claim holds for all~$\delta<\min(\eta,\mu)$ (where the bound on~$\delta$ ensures that the induction statement is primitive recursive). Concerning clause~(iv) of the previous definition, we note that the induction hypothesis applies since Lemma~\ref{lem:sep-pred} ensures that we get $\otp(\{D\}^\alpha(\omega^{1+\eta}))<\xi$ for~$\alpha\leq\delta<\eta$.
\end{proof}

We can deduce that $\mathbb J$ is a partial function.

\begin{proposition}[$\atrs$]\label{prop:J-partial-fct}
Given $\mathbb J(D,\gamma)\simeq\delta$ and $\mathbb J(D,\gamma)\simeq\delta'$, we get $\delta=\delta'$.
\end{proposition}
\begin{proof}
Due to the previous lemma, we can assume $\mathbb J(D,\gamma)\simeq^\eta_\xi\delta$ and $\mathbb J(D,\gamma)\simeq^\eta_\xi\delta'$ for common~$\eta,\xi$. By guarded induction on~$D$, one readily derives~$\delta=\delta'$.
\end{proof}

In particular, there is a unique function~$f$ as in clause~(iii) of Definition~\ref{def:J} as soon as a value~$\delta_\nu\leq\delta$ with $\mathbb J(\sum_{\alpha<\nu}D_\alpha,\gamma)\simeq^\eta_\xi\delta_\nu$ exists for each~$\nu<\lambda$. Note that $\Delta_0$-collection (which is available in Kripke-Platek set theory but not in~$\atrs$) may be needed to find common bounds~$\delta,\eta,\xi$ if they are not already given. 

The following provides further confirmation that we have managed to capture the original definition of~$\mathbb J$ by P\"appinghaus, which was given in the introduction.

\begin{proposition}[$\atrs$]\label{prop:clauses-J}
For a dilator~$D$ and ordinals~$\gamma,\delta$, we have $\mathbb J(D,\gamma)\simeq\delta$ precisely when one of the following applies:
\begin{enumerate}[label=(\roman*)]
\item $D$ has type~$0$ and we have $\delta=\gamma$,
\item $D\cong D'+1$ has type~$1$ and~$\delta=\delta'+1$ is a successor with~$\mathbb J(D',\gamma)\simeq\delta'$,
\item $D\cong\sum_{\alpha<\lambda}D_\alpha$ has type~$\omega$ (with connected~$D_\alpha$) and $\delta=\sup_{\nu<\lambda}f(\nu)$ holds for a function~$f:\lambda\to\delta+1$ with $\mathbb J(\sum_{\alpha<\nu}D_\alpha,\gamma)\simeq f(\nu)$ for all~$\nu<\lambda$,
\item $D$ has type~$\Omega$ and we have $\delta=\alpha+\beta$ for (necessarily unique) ordinals~$\alpha,\beta$ such that we have $\mathbb J(\{D\}^0,\gamma)\simeq\alpha$ and $\mathbb J(\{D\}^\alpha,\gamma)\simeq\beta$.
\end{enumerate}
\end{proposition}
\begin{proof}
One distinguishes cases according to the type of~$D$. As a representative~example, we assume~(iii) and derive~$\mathbb J(D,\gamma)\simeq\delta$. We are thus given
\begin{equation*}
\mathbb J\big(\textstyle\sum_{\alpha<\nu}D_\alpha,\gamma\big)\simeq^{\eta(\nu)}_{\xi(\nu)}f(\nu)\leq\delta
\end{equation*}
for all~$\nu<\lambda$ and suitable~$\eta(\nu),\xi(\nu)$. Set $\eta=\delta+1$ and~$\xi=\otp(D(\omega^{1+\eta}))$, where the latter exists by axiom beta. Due to Lemma~\ref{lem:J-indep}, we may assume that we have $\eta(\nu)=\eta$ and $\xi(\nu)=\xi$ for all~$\nu<\lambda$. Now Definition~\ref{def:J} yields~$\mathbb J(D,\gamma)\simeq^\eta_\xi\delta$.
\end{proof}

Next, we show that~$\atrs$ proves basic properties of~$\mathbb J$ (originally due to P\"apping\-haus) under the assumption that the relevant values of~$\mathbb J$ exist. The~following result corresponds to Lemma~1.4 of~\cite{paeppinghaus-bachmann}. Let us recall that we write $D\leq D'$ for dilators~$D$ and~$D'$ if there is a natural transformation~$D\Rightarrow D'$. In statement~(a) of the following lemma, we write $\mathbb J(D,\gamma)\lesssim\mathbb J(D',\gamma')$ to assert that if $\mathbb J(D',\gamma')\simeq\delta'$ is defined then so is $\mathbb J(D,\gamma)\simeq\delta$ with $\delta\leq\delta'$. Statement~(b) means that $\mathbb J(D+E,\gamma)\simeq\delta$ holds precisely if there is a (necessarily unique) ordinal~$\alpha$ that validates $\mathbb J(D,\gamma)\simeq\alpha$ and $\mathbb J(E,\alpha)\simeq\delta$. Similar notation is well-established in the context of partial functions on the natural numbers.

\begin{lemma}[$\atrs$]\label{lem:J-nat-trans}
(a) Given $\gamma\leq\gamma'$ and~$D\leq D'$, we get $\mathbb J(D,\gamma)\lesssim\mathbb J(D',\gamma')$.

(b) We have $\mathbb J(D+E,\gamma)\simeq\mathbb J(E,\mathbb J(D,\gamma))$.
\end{lemma}
\begin{proof}
(a) We cannot immediately use induction on~$D'$, as the condition~\mbox{$D\leq D'$} is not primitive recursive. In order to resolve this issue, we fix~$D,D'$ and ordinals that validate~$\mathbb J(D',\gamma')\simeq^\eta_\xi\delta'$ (note that~$\eta$ occurs in~(iii) and~(iv) below). Using primitive recursion along~$\omega$ (cf.~the proof of Proposition~\ref{prop:pred-ind}), we construct a set~$T$ of natural transformations that contains a witness for~$D\leq D'$ and has the following closure properties, which rely on Proposition~\ref{prop:sums}(b) and Lemma~\ref{lem:Omega-nat-trans}:
\begin{enumerate}[label=(\roman*)]
\item If $T$ contains a natural transformation that witnesses~$E\leq F+1$, it contains a witness for~$E\leq F$ or a witness for~$E'\leq F$ with $E\cong E'+1$.
\item If $T$ contains a witness for~$E\leq\sum_{\alpha<\lambda}D_\alpha$ (with limit~$\lambda$ and connected~$D_\alpha$), it contains one for~$E\leq\sum_{\alpha<\nu}D_\alpha$ with $\nu<\lambda$ or we have~$E\cong\sum_{\alpha<\kappa}E_\alpha$ (with limit~$\kappa$ and connected~$E_\alpha$) and every~$\mu<\kappa$ admits a~$\nu<\lambda$ such that $T$ contains a witness for~$\sum_{\alpha<\mu}E_\alpha\leq\sum_{\alpha<\nu}D_\alpha$.
\item If $T$ contains a witness for~$E\leq F$ with~$F$ of type~$\Omega$, it contains a witness for~$E\leq\{F\}^0$ or the dilator~$E$ does also have type~$\Omega$ and our set~$T$ contains witnesses for~$\{E\}^\alpha\leq\{F\}^\alpha$ for all~$\alpha<\eta$.
\item If $E$ has type~$\Omega$ and occurs in~$T$ (i.\,e., if $T$ contains a witness for~$E_0\leq E_1$ where~$E=E_i$), then~$T$ has witnesses for~$\{E\}^\alpha\leq\{E\}^\beta$ for all~$\alpha\leq\beta<\eta$.
\end{enumerate}
We now use guarded induction on~$F$ to show that~$\mathbb J(F,\beta)\simeq^\eta_\xi\tau$ entails~$\mathbb J(E,\alpha)\simeq^\eta_\xi\sigma$ with~$\sigma\leq\tau$ whenever we have~$\alpha\leq\beta$ and~$T$ contains a witness for~$E\leq F$. An analogous induction is carried out in the proof of Lemma~1.4 from~\cite{paeppinghaus-bachmann}. To indicate the issues that are specific to our setting, we discuss the case where~$F$ has type~$\Omega$ and we have~$E\not\leq\{F\}^0$. Consider~$\tau=\pi+\rho$ with
\begin{equation*}
\mathbb J(\{F\}^0,\beta)\simeq^\eta_\xi\pi\quad\text{and}\quad\mathbb J(\{F\}^\pi,\beta)\simeq^\eta_\xi\rho.
\end{equation*}
By~(iii), our set~$T$ contains a witness for~$\{E\}^0\leq\{F\}^0$. The induction hypothesis thus yields a~$\chi\leq\pi$ with $\mathbb J(\{E\}^0,\alpha)\simeq^\eta_\xi\chi$. We also have~$\{E\}^\chi\leq\{E\}^\pi\leq\{F\}^\pi$ with witnesses in~$T$, due to~(iii) and~(iv). This yields a~$\rho'\leq\rho$ with $\mathbb J(\{E\}^\chi,\alpha)\simeq^\eta_\xi\rho'$. Since~$E\leq F$ does also entail
\begin{equation*}
\otp(E(\omega^{1+\eta}))\leq\otp(F(\omega^{1+\eta}))<\xi,
\end{equation*}
we finally get $\mathbb J(E,\alpha)\simeq^\eta_\xi\sigma$ with $\sigma=\chi+\rho'\leq\pi+\rho=\tau$.

(b) For fixed~$D$ and~$\gamma,\eta,\xi$, guarded induction on~$E$ with $\otp((D+E)(\omega^{1+\eta}))<\xi$ shows that $\mathbb J(D+E,\gamma)\simeq^\eta_\xi\delta$ holds precisely if there is an~$\alpha$ with $\mathbb J(D,\gamma)\simeq^\eta_\xi\alpha$ and~$\mathbb J(E,\alpha)\simeq^\eta_\xi\delta$ (with a bounded quantification over $\alpha,\delta<\eta$). As an example, we discuss the backward direction for the case where we have~$E\cong\sum_{\beta<\lambda}E_\beta$ (with a limit ordinal~$\lambda$ and connected~$E_\beta$). For the other cases, we refer to the proof of Lemma~1.4 in~\cite{paeppinghaus-bachmann}. By Proposition~\ref{prop:sums}(a), we get $D\cong\sum_{\beta<\kappa}D_\beta$ with connected~$D_\beta$ (where~$\kappa$ need not be limit). Let us write
\begin{equation*}
D+E=\textstyle\sum_{\beta<\kappa+\lambda}F_\beta\quad\text{for}\quad F_\beta=\begin{cases}
D_\beta & \text{if $\beta<\kappa$},\\
E_{\beta'} & \text{if $\beta=\kappa+\beta'$ with $\beta'<\lambda$}.
\end{cases}
\end{equation*}
We assume~$\mathbb J(D,\gamma)\simeq^\eta_\xi\alpha$ and~$\mathbb J(E,\alpha)\simeq^\eta_\xi\delta$. The latter provides an~$f:\lambda\to\delta+1$ with $\delta=\sup_{\nu<\lambda}f(\nu)$ and~$\mathbb J(\sum_{\beta<\nu}E_\beta,\alpha)\simeq^\eta_\xi f(\nu)$ for all~$\nu<\lambda$. By the induction hypothesis, we obtain
\begin{equation*}
\mathbb J\big(\textstyle\sum_{\beta<\kappa+\nu}F_\beta,\gamma\big)\simeq^\eta_\xi f(\nu).
\end{equation*}
For~$\mu<\kappa$, part~(a) yields values~$g(\mu)\leq\alpha=f(0)$ with $\mathbb J(\textstyle\sum_{\beta<\mu}F_\beta,\gamma)\simeq^\eta_\xi g(\mu)$. Define~$h:\kappa+\lambda\to\delta+1$ by~$h(\mu)=g(\mu)$ for~$\mu<\kappa$ and $h(\kappa+\nu)=f(\nu)$ for~$\nu<\lambda$. We then have $\delta=\sup_{\nu<\kappa+\lambda}h(\nu)$ and $\mathbb J(\textstyle\sum_{\beta<\nu}F_\beta,\gamma)\simeq^\eta_\xi h(\nu)$ for all~$\nu<\kappa+\lambda$. Since the condition $\otp((D+E)(\omega^{1+\eta}))<\xi$ from Definition~\ref{def:J} was explicitly assumed in the present induction, we finally get~$\mathbb J(D+E,\gamma)\simeq^\eta_\xi\delta$.
\end{proof}

The following corresponds to Lemmas~1.5 and~1.6 of~\cite{paeppinghaus-bachmann}.

\begin{lemma}[$\atrs$]\label{lem:J-properties}
(a) We have $\gamma\lesssim\mathbb J(D,\gamma)$.

(b) Given $1\leq D$, we get $\gamma+1\lesssim\mathbb J(D,\gamma)$.

(c) If we have $1\leq D$ and~$D$ has type~$\Omega$, we get $\mathbb J(\{D\}^{\gamma+1},\gamma)\lesssim\mathbb J(D,\gamma)$.

(d) For $D\neq 0$, we have $\gamma+1\lesssim\mathbb J(D,\gamma)$ or~$\mathbb J(D,\gamma)\simeq\gamma=0$.

(e) Given $\mathbb J(D+E,\gamma)\simeq\delta\neq 0$ with $E\neq 0$, we get $\mathbb J(D,\gamma)\simeq\delta'$ with $\delta'<\delta$.
\end{lemma}
\begin{proof}
To establish~(a) and~(b), one uses part~(a) of the previous lemma. This relies on the fact that we always have~$0\leq D$ as well as $\mathbb J(0,\gamma)\simeq\gamma$ and $\mathbb J(1,\gamma)\simeq\gamma+1$.

For~(c), we assume we have~$\mathbb J(D,\gamma)\simeq\delta$. By Proposition~\ref{prop:clauses-J} we get $\delta=\alpha+\beta$ with $\mathbb J(\{D\}^0,\gamma)\simeq\alpha$ and~$\mathbb J(\{D\}^\alpha,\gamma)\simeq\beta$. Lemma~\ref{lem:Omega-nat-trans} yields~$1\leq\{D\}^0$, so that we get $\gamma+1\leq\alpha$ by part~(b). Due to Lemma~\ref{lem:Omega-nat-trans}, we obtain $\{D\}^{\gamma+1}\leq\{D\}^\alpha$. Again by the previous lemma, we finally get~$\mathbb J(\{D\}^{\gamma+1},\gamma)\simeq\beta_0$ with $\beta_0\leq\beta\leq\delta$.

Concerning~(d), we note that $\gamma+1\lesssim\mathbb J(D,\gamma)$ is automatically true when~$\mathbb J(D,\gamma)$ is undefined. So assume ~$\mathbb J(D,\gamma)\simeq^\eta_\xi\delta$. By guarded induction on~$E\neq 0$, we show
\begin{equation*}
\mathbb J(E,\gamma)\simeq^\eta_\xi\delta'\quad\Rightarrow\quad\gamma+1\leq\delta'\text{ or }\delta'=\gamma=0.
\end{equation*}
When~$E$ has type~$1$, the claim follows from~(b). The induction step is straightforward for~$E$ of type~$\omega$. For~$E$ of type~$\Omega$, the antecedent of the desired implication yields~$\delta'=\alpha+\beta$ with $\mathbb J(\{D\}^0,\gamma)\simeq\alpha$ and~$\mathbb J(\{D\}^\alpha,\gamma)\simeq\beta$. In view of part~(a), we have~$\gamma\leq\alpha,\beta$. If we have $\gamma<\alpha$ or~$0<\beta$, we get $\gamma+1\leq\delta'$. Otherwise, we can conclude $\alpha=\beta=\gamma=0$ and hence~$\delta'=0$.

For~(e), use the previous lemma to get a~$\delta'$ with $\mathbb J(D,\gamma)\simeq\delta'$ and~$\mathbb J(E,\delta')\simeq\delta$. Given~$\delta\neq 0$, we must have~$\delta'<\delta$ by part~(d).
\end{proof}

We conclude this section with the metamathematical result that was promised above. In the following theorem, $\Pi^1_1$-comprehension is considered under the usual interpretation of second-order arithmetic in set theory. Our base theory~$\atrs$ (with the axiom of countability) proves that $\Pi^1_1$-comprehension holds precisely if every set is contained in some admissible set, i.\,e., in a transitive model of Kripke-Platek set theory (see Section~7 of~\cite{jaeger-admissibles} or Corollary~1.4.13 of~\cite{freund-thesis}).

\begin{theorem}[$\atrs$]\label{thm:Pi11CA-to-J}
If $\Pi^1_1$-comprehension holds, then~$\mathbb J$ is total, i.\,e., for all dilators~$D$ and ordinals~$\gamma$, there is an ordinal~$\delta$ with $\mathbb J(D,\gamma)\simeq\delta$.
\end{theorem}
\begin{proof}
Given $\Pi^1_1$-comprehension, it suffices to establish the conclusion for~$D$ and~$\gamma$ that lie in an admissible set~$\mathbb A\ni\omega$ (see the paragraph before the theorem). Consider
\begin{equation*}
\mathbb J\restriction\mathbb A=\big\{(D,\gamma,\delta)\in\mathbb A^3\,\big|\,\mathbb J(D,\gamma)\simeq^\eta_\xi\delta\text{ for some }\eta,\xi\in\mathbb A\big\}.
\end{equation*}
We recall that primitive recursive set functions are total and $\Sigma$-definable in admissible sets (with a single $\Sigma$-definition over the language~$\{\in\}$; see Lemma~2.4.1 of~\cite{freund-thesis} for an explicit reference for this known result). Hence $\mathbb J\restriction\mathbb A$ is a $\Sigma$-class from the viewpoint of~$\mathbb A$ (but a set in the ambient universe). As a crucial preparation, we show that
\begin{equation*}
(D,\gamma,\delta)\in\mathbb J\restriction\mathbb A\quad\Leftrightarrow\quad\mathbb J(D,\gamma)\simeq\delta
\end{equation*}
holds for~$D,\gamma,\delta\in\mathbb A$ such that $D$ is a dilator (in the sense of the universe). The forward direction is immediate. For the backward direction, we note that the right side yields~$\mathbb J(D,\gamma)\simeq^\eta_\xi\delta$ with $\eta=\delta+1$ and $\xi=\otp(D(\omega^{1+\eta}))+1$, by Lemma~\ref{lem:J-indep}. Given that $\mathbb A$ contains~$D$ and $\delta$, it also contains~$\eta$ and $D(\omega^{1+\eta})=\overline D(\omega^{1+\eta})$, the latter because~$\overline D:\lo\to\lo$ is primitive recursive with parameter~$D:\nat\to\lo$ (see~Definition~\ref{def:dil-extend} and the paragraph that follows it). As $D(\omega^{1+\eta})\in\mathbb A$ is a well order in the sense of the universe, its order type also lies in~$\mathbb A$ (see Theorem~4.6 of~\cite{jaeger-admissibles}). So we finally get~$\xi\in\mathbb A$, which yields $(D,\gamma,\delta)\in\mathbb J\restriction\mathbb A$.

We now use guarded induction on dilators~$D\in\mathbb A$ to show that any~$\gamma\in\mathbb A$ admits a~$\delta\in\mathbb A$ with $\mathbb J(D,\gamma)\simeq\delta$. The latter condition is primitive recursive due to the equivalence above. Note that the induction is guarded because we will only consider predecessors from~$P_\Omega(D)$, where $\Omega$ is the intersection of~$\mathbb A$ with the class of ordinals. Let us point out that $D\in\mathbb A$ entails~$P_\Omega(D)\subseteq\mathbb A$ (see Definition~\ref{def:dil-preds}), as the construction of predecessors is primitive recursive. So in the induction step for~$D,\gamma\in\mathbb A$, the induction hypothesis says that each~$D'\in P_\Omega(D)$ admits a~$\delta\in\mathbb A$ with $\mathbb J(D',\gamma)\simeq\delta$. The cases where $D$ has type~$0$ or~$1$ are immediate by parts~(i) and~(ii) of Proposition~\ref{prop:clauses-J}.

Now consider~$D\cong\sum_{\alpha<\lambda}D_\alpha$ with limit~$\lambda$ and connected~$D_\alpha$. We note that $\lambda$ is primitive recursive in~$D$ and hence contained in~$\mathbb A$. By the induction hypothesis, each $\nu<\lambda$ admits a~$\delta'\in\mathbb A$ with $\mathbb J(\sum_{\alpha<\nu}D_\alpha,\gamma)\simeq\delta'$. As we have seen, the latter is equivalent to~$(\sum_{\alpha<\nu}D_\alpha,\gamma,\delta')\in\mathbb J\restriction\mathbb A$, so that we are concerned with a $\Sigma$-property from the viewpoint of~$\mathbb A$. Let us also note that the~$\delta'$ are unique by Proposition~\ref{prop:J-partial-fct}. As Kripke-Platek set theory proves $\Sigma$-replacement (see Theorem~I.4.6 of~\cite{barwise-admissible}), we get a function~$f\in\mathbb A$ with~$\mathbb J(\sum_{\alpha<\nu}D_\alpha,\gamma)\simeq f(\nu)$ for~$\nu<\lambda$. Put $\delta=\sup_{\nu<\lambda}f(\nu)$ and consider~$f$ as a function from~$\lambda$ into~$\delta+1$. By part~(iii) of Proposition~\ref{prop:clauses-J}, we obtain $\mathbb J(D,\gamma)\simeq\delta\in\mathbb A$, as required.

Finally, we assume that~$D$ has type~$\Omega$. By the induction hypothesis, we get $\mathbb J(\{D\}^0,\gamma)\simeq\alpha$ and then $\mathbb J(\{D\}^\alpha,\gamma)\simeq\beta$ with $\alpha,\beta\in\mathbb A$. Part~(iv) of Proposition~\ref{prop:clauses-J} tells us that $\delta=\alpha+\beta\in\mathbb A$ validates $\mathbb J(D,\gamma)\simeq\delta$.
\end{proof}

\section{Recursive Clauses for Bachmann-Howard Fixed Points}\label{sect:Bachmann-rec}

As mentioned in the introduction, the axiom of $\Pi^1_1$-comprehension is characterized by a map~$D\mapsto\psi D$ that transforms a dilator into an ordinal. The orders~$\psi D$ were originally defined by a simple explicit construction. In this section, we show that~$\psi D$ can also be defined by recursion over the dilator~$D$. While this definition is more complicated, it facilitates comparisons with other recursive~constructions.

The original construction of~$\psi$ (see Definition~1.4 of~\cite{freund-rathjen-iterated-Pi11}) yields more general orders~$\psi_\nu D$, which involve iterations of collapsing along an infinite ordinal~$\nu$. In the special case of $\psi D=\psi_1 D$, the following definition is simpler but equivalent over~$\prs$ (as one readily derives from Lemma~4.2 of~\cite{uftring-inverse-goodstein}). 

\begin{definition}\label{def:psi-fp}
An ordinal~$\Omega$ is called a $\psi$-fixed point of a predilator~$D$ if there is an embedding~$\pi:\Omega\to D(\Omega)$ such that all $\alpha,\beta<\Omega$ and all $\tau\in D(\Omega)$ validate
\begin{align}
\alpha\in\supp_\Omega(\pi(\beta))\quad&\Rightarrow\quad\alpha<\beta,\label{psi-fp1}\\
\pi(\alpha)<\tau\text{ for all }\alpha\in\supp_\Omega(\tau)\quad&\Rightarrow\quad\tau\in\rng(\pi).\label{psi-fp2}
\end{align}
In view of the following lemma, we write $\psi D$ for the $\psi$-fixed point of~$D$, if it exists.
\end{definition}

The following is an immediate consequence of Proposition~2.1 from~\cite{freund-rathjen-iterated-Pi11}. The latter also shows that the map~$\pi$ in the previous definition is uniquely determined.

\begin{lemma}[$\prs$]\label{lem:psi-unique}
Each dilator has at most one~$\psi$-fixed point.
\end{lemma}

The existence of $\psi$-fixed points for all dilators is equivalent to $\Pi^1_1$-comprehension over $\atrs$. A finer picture is established in~\cite{freund-rathjen-iterated-Pi11}: Computably in~$D$, one can construct a linear order that validates the defining conditions for~$\psi D$. The principle of $\Pi^1_1$-comprehension is only needed to show that this order is well-founded (and hence isomorphic to an ordinal by axiom beta). In view of Lemma~\ref{lem:otp-restr}, we can test whether an ordinal is isomorphic to the computable representation of~$\psi D$ that we have just mentioned. This gives the following result, which will be important later.

\begin{lemma}[$\prs$]\label{lem:psi-prim-rec}
The class
\begin{equation*}
\{(D,\gamma)\,|\,\text{``$D$ is a predilator with $\psi$-fixed point~$\gamma$"}\}
\end{equation*}
is primitive recursive in the sense of Jensen and Karp (cf.~Section~\ref{sect:prs}).
\end{lemma}

In our recursive characterization of $\psi$-fixed points, we will need to consider the following variants of a given dilator.

\begin{definition}\label{def:D-gamma}
For a dilator~$D$ and an ordinal~$\gamma$, we define $D^\gamma$ as the dilator that is given by~$D^\gamma(\alpha)=D(\gamma+\alpha)$, where the action on morphisms and the support functions are explained as in Definition~\ref{def:sep-vars} (but without the condition on~$i(\sigma)$). We write $D[\gamma]$ for~$D^\gamma$ when we wish to save subscripts.
\end{definition}

Let us note that $D^0$ coincides with~$D$.

\begin{remark}
One can obtain a recursive characterization of $\psi D^\gamma$ by combining results from the literature. Write $\nu(D,\gamma)$ for the length of the increasing Goodstein sequence for $D$ with start value $\gamma$, as considered by Abrusci, Girard and van de Wiele~\cite{increasing-goodstein}. It was suggested by Weiermann that increasing Goodstein sequences are connected to Bachmann-Howard fixed points, so that one obtains an equivalence with $\Pi^1_1$-comprehension. This was confirmed by Uftring~\cite{uftring-inverse-goodstein}, who proved the remarkable equation $\nu(D,0)=\psi D$. We get
\begin{equation*}
\nu(D,\gamma)=\gamma+\nu(D^\gamma,0) =\gamma+\psi D^\gamma.
\end{equation*}
Here the first equality holds by~\cite{increasing-goodstein}, which also provides a definition of $\nu(D,\gamma)$ by recursion over the dilator~$D$, simultaneously for all~$\gamma$. While this gives recursive clauses for~$\psi D^\gamma$, these are obtained in a somewhat indirect way. In the following, we give a direct proof and check that it goes through in our base theory~$\prs$.
\end{remark}

We first determine the~$\psi$-fixed points of the most simple dilators.

\begin{lemma}[$\prs$]\label{lem:psi-constant}
For a constant dilator~$D=\alpha$ and any~$\gamma$, we have $\psi D^\gamma=\alpha$.
\end{lemma}
\begin{proof}
The dilator~$D^\gamma$ is also constant with value~$\alpha$. Its support sets are thus empty (by naturality). This trivializes~(\ref{psi-fp1}) and the premise of~(\ref{psi-fp2}). By the conclusion of the latter, our embedding~$\pi:\psi D^\gamma\to D^\gamma(\psi D^\gamma)=\alpha$ must be surjective.
\end{proof}

An arbitrary dilator~$D$ can be uniquely written as a sum~$\sum_{\alpha<\mu}D_\alpha$ of connected dilators~$D_\alpha$ (see~Definition~\ref{def:dil-types}). When~$\mu$ is a successor, $D$ has type~$1$ or~$\Omega$. In both these cases, our recursion relies on the following result (cf.~Lemma~\ref{lem:J-nat-trans}).

\begin{proposition}[$\prs$]\label{prop:psi-sum}
For all dilators~$D,E$ and any ordinal~$\gamma$, we have
\begin{equation*}
\psi(D+E)^\gamma=\psi D^\gamma+\psi E^\delta\quad\text{with}\quad\delta=\gamma+\psi D^\gamma.
\end{equation*}
More precisely, if the $\psi$-fixed points of $D^\gamma$ and $E^\delta$ exist, then their sum provides the $\psi$-fixed point of~$(D+E)^\gamma$. 
\end{proposition}
\begin{proof}
The given $\psi$-fixed points come with embeddings
\begin{equation*}
\pi_D:\psi D^\gamma\to D^\gamma(\psi D^\gamma)\quad\text{and}\quad\pi_E:\psi E^\delta\to E^\delta(\psi E^\delta).
\end{equation*}
We combine these into a suitable embedding
\begin{equation*}
\pi:\psi D^\gamma+\psi E^\delta\to (D+E)^\gamma(\underbrace{\psi D^\gamma+\psi E^\delta}_{=:\eta})=D^\gamma(\eta)+E^\delta(\psi E^\delta).
\end{equation*}
For the inclusion $\iota:\psi D^\gamma\hookrightarrow\eta$, we set
\begin{equation*}
\pi(\alpha)=\begin{cases}
D^\gamma(\iota)(\pi_D(\alpha)) & \text{if }\alpha<\psi D^\gamma,\\
D^\gamma(\eta)+\pi_E(\alpha') & \text{if }\alpha=\psi D^\gamma+\alpha'.
\end{cases}
\end{equation*}
Write $\supp^{F,\gamma}_\alpha:F^\gamma(\alpha)\to[\alpha]^{<\omega}$ for the support functions of a dilator~$F^\gamma$. We have
\begin{equation*}
\supp^{D+E,\gamma}_\eta(\tau)=\begin{cases}
\supp^{D,\gamma}_\eta(\tau) & \text{if }\tau<D^\gamma(\eta),\\[.5ex]
\supp^{E,\gamma}_\eta(\tau') & \text{if }\tau=D^\gamma(\eta)+\tau'.
\end{cases}
\end{equation*}
In view of $\gamma+\eta=\delta+\psi E[\delta]$ (recall $E[\delta]:=E^\delta$), we also get
\begin{equation*}
\psi D^\gamma+\alpha'\in\supp^{E,\gamma}_\eta(\tau')\quad\Leftrightarrow\quad\alpha'\in\supp^{E,\delta}_{\psi E[\delta]}(\tau').
\end{equation*}
We now verify the conditions from Definition~\ref{def:psi-fp}. The first of these reads
\begin{equation*}
\alpha\in\supp^{D+E,\gamma}_\eta(\pi(\beta))\quad\Rightarrow\quad\alpha<\beta.
\end{equation*}
When we have $\beta<\psi D^\gamma$, the premise of this implication amounts to
\begin{equation*}
\alpha\in\supp^{D,\gamma}_\eta(D^\gamma(\iota)(\pi_D(\beta)))=\supp^{D,\gamma}_{\psi D[\gamma]}(\pi_D(\beta)),
\end{equation*}
where the last equality holds since supports are natural. We now obtain $\alpha<\beta$ by condition~(\ref{psi-fp1}) for~$\pi_D$. When we have $\beta=\psi D^\gamma+\beta'$, the premise of the desired implication is equivalent to
\begin{equation*}
\alpha\in\supp^{E,\gamma}_\eta(\pi_E(\beta')).
\end{equation*}
If we have $\alpha<\psi D^\gamma$, we immediately get the desired conclusion~$\alpha<\beta$. Otherwise, we may write $\alpha=\psi D^\gamma+\alpha'$ to get
\begin{equation*}
\alpha'\in\supp^{E,\delta}_{\psi E[\delta]}(\pi_E(\beta')).
\end{equation*}
Due to~(\ref{psi-fp1}) for~$\pi_E$, this yields $\alpha'<\beta'$ and hence $\alpha<\beta$, as desired. The remaining condition demands
\begin{equation*}
\pi(\alpha)<\tau\text{ for all }\alpha\in\supp^{D+E,\gamma}_\eta(\tau)\quad\Rightarrow\quad\tau\in\rng(\pi).
\end{equation*}
First assume we have $\tau<D^\gamma(\eta)$. Let us note that $\pi(\alpha)<D^\gamma(\eta)$ forces $\alpha<\psi D^\gamma$. 
Given the premise of the desired implication, we thus obtain
\begin{equation*}
\supp^{D,\gamma}_\eta(\tau)\subseteq\rng(\iota).
\end{equation*}
This allows us to write $\tau=D^\gamma(\iota)(\tau_0)$, due to the support condition from the~definition of dilators. We get $\pi(\alpha)=D^\gamma(\iota)(\pi_D(\alpha))<\tau$ and hence $\pi_D(\alpha)<\tau_0$ for~all
\begin{equation*}
\alpha\in\supp^{D+E,\gamma}_\eta(\tau)=\supp^{D,\gamma}_{\psi D[\gamma]}(\tau_0).
\end{equation*}
Condition~(\ref{psi-fp2}) for~$\pi_D$ yields $\tau_0=\pi_D(\beta)$ with $\beta<\psi D^\gamma$. We get $\tau=\pi(\beta)\in\rng(\pi)$, as desired. Finally, we consider~$\tau=D^\gamma(\eta)+\tau'$. Given the premise of the desired implication, we see that any~$\alpha'$ validates
\begin{align*}
\alpha'\in\supp^{E,\delta}_{\psi E[\delta]}(\tau')\quad&\Rightarrow\quad\psi D^\gamma+\alpha'\in\supp^{D+E,\gamma}_\eta(\tau)\\
{}&\Rightarrow\quad\pi(\psi D^\gamma+\alpha')<\tau\quad\Rightarrow\quad\pi_E(\alpha')<\tau'.
\end{align*}
By condition~(\ref{psi-fp2}) for~$\pi_E$, we may write~$\tau'=\pi_E(\beta)$ and thus~$\tau=\pi(\psi D^\gamma+\beta)$.
\end{proof}

The following lemma (or rather its proof) will help us to analyse sums of limit length. Recall that for dilators, $D\leq E$ denotes the existence of a natural transformation. Given~$D\leq E$, we clearly get $D^\gamma\leq E^\gamma$ for any ordinal~$\gamma$ (cf.~Lemma~\ref{lem:Omega-nat-trans}).

\begin{lemma}[$\prs$]
For any dilators~$D\leq E$, we have $\psi D\leq\psi E$. More precisely, if~$\psi E$ exists, then $\psi D$ exists and the indicated inequality holds.
\end{lemma}
\begin{proof}
Consider a natural transformation~$\eta:D\Rightarrow E$ and a map~$\pi_E:\psi E\to E(\psi E)$ as in the definition of $\psi$-fixed points. We stipulate
\begin{equation*}
\Psi=\left\{\alpha<\psi E\,\left|\,\pi_E(\alpha)\in\rng\big(\eta_{\psi E}\big)\text{ and }\supp^E_{\psi E}(\pi_E(\alpha))\subseteq\Psi\right.\right\}.
\end{equation*}
Note that $\alpha\in\Psi$ can be decided by recursion on~$\alpha$, as condition~(i) of Definition~\ref{def:psi-fp} guarantees that $\alpha'\in\supp^E_{\psi(E)}(\pi(\alpha))$ entails~$\alpha'<\alpha$. By Lemma~\ref{lem:otp-restr}, we have an ordinal~$\psi D$ that admits an embedding~$f:\psi D\to\psi E$ with range~$\Psi$. To justify the notation, we show that~$\psi D$ is indeed our $\psi$-fixed point. We claim that there are embeddings~$\pi_D$ and~$g$ that make
\begin{equation*}
\begin{tikzcd}
\psi D\arrow[rr,"f"]\arrow[rd,dashed,"g"]\arrow[d,dashed,swap,"\pi_D"] & & \psi E\arrow[d,"\pi_E"]\\
D(\psi D)\arrow[r,"D(f)"]\arrow[rr,swap,bend right=20,"E(f)\circ\eta_{\psi D}"] & D(\psi E)\arrow[r,"\eta_{\psi E}"] & E(\psi E)
\end{tikzcd}
\end{equation*}
a commutative diagram. Let us first note that the curved arrow, which has been added for later reference, arises from the naturality of~$\eta$. The existence of a $g$ for which the inner quadrangle commutes is guaranteed by the first conjunct from the definition of~$\Psi$. For the triangle, we recall that any natural transformation between dilators preserves supports (see~\cite{girard-pi2} or Lemma~2.19 of~\cite{freund-rathjen_derivatives}), which yields the first equality in
\begin{equation*}
\supp^D_{\psi E}(g(\alpha))=\supp^E_{\psi E}\left(\eta_{\psi E}\circ g(\alpha)\right)=\supp^E_{\psi E}(\pi_E\circ f(\alpha))\subseteq\Psi=\rng(f).
\end{equation*}
Here the inclusion holds by the second conjuct from the definition of~$\Psi$. Now the support condition for dilators (see Definition~\ref{def:predil}) yields~$g(\alpha)\in\rng(D(f))$, as needed to find~$\pi_D$ in the diagram. One can verify that $\pi_D$ satisfies the conditions from Definition~\ref{def:psi-fp}.
\end{proof}

As promised, we now consider sums of limit length.

\begin{proposition}[$\atrs$]\label{prop:psi-limit-sum}
When~$\lambda$ is a limit ordinal and~$D_\alpha$ for~$\alpha<\lambda$ are dilators, then any~$\gamma$ validates
\begin{equation*}
\psi\big(\textstyle\sum_{\alpha<\lambda} D_\alpha\big)^\gamma=\textstyle\sup_{\mu<\lambda}\psi\big(\textstyle\sum_{\alpha<\mu} D_\alpha\big)^\gamma.
\end{equation*}
More precisely, if the $\psi$-fixed points of $(\sum_{\alpha<\mu} D_\alpha)^\gamma$ for~$\mu<\lambda$ exist, their supremum is the $\psi$-fixed point of~$(\sum_{\alpha<\lambda} D_\alpha)^\gamma$.
\end{proposition}
Let us note that the supremum in the last sentence can be obtained in~$\atrs$. This is because the orders~$\psi(\textstyle\sum_{\alpha<\mu} D_\alpha)^\gamma$ have uniformly computable representations (see the paragraph after Lemma~\ref{lem:psi-unique}). Given that these are well-founded, their sum is isomorphic to an ordinal that bounds the supremum, due to axiom beta. The supremum itself is then provided by Lemma~\ref{lem:otp-restr}.
\begin{proof}
We abbreviate~$(\sum_{\alpha<\mu}D_\alpha)^\gamma$ as $D_{<\mu}$ or $D[<\mu]$ (the latter to save subscripts) and write $\supp^\mu_\beta:D_{<\mu}(\beta)\to[\beta]^{<\omega}$ for the associated support functions. Let
\begin{equation*}
\pi_\mu:\psi D_{<\mu}\to D_{<\mu}\big(\psi D_{<\mu}\big)
\end{equation*}
witness that we are concerned with $\psi$-fixed points. The maps~$\pi_\mu$ are encoded in the computable representations that were mentioned in the paragraph before this proof. It follows that the family of maps~$\pi_\mu$ for~$\mu<\lambda$ is available in~$\atrs$.

To glue the maps~$\pi_\mu$, we need to show that they are compatible. For~$\mu<\nu\leq\lambda$, let~$\eta^{\mu,\nu}:D_{<\mu}\Rightarrow D_{<\nu}$ be the obvious natural transformation. Now assume that we have~$\mu<\nu<\lambda$. By Proposition~2.1 of~\cite{freund-rathjen-iterated-Pi11}, not only the order~$\psi D_{<\mu}$ but also the embedding~$\pi_\mu$ is uniquely determined by the conditions from Definition~\ref{def:psi-fp}. In view of the previous proof, the diagram
\begin{equation*}
\begin{tikzcd}
\psi D_{<\mu}\arrow[rrr,"f^{\mu,\nu}"]\arrow[d,swap,"\pi_\mu"] & & & \psi D_{<\nu}\arrow[d,"\pi_\nu"]\\
D_{<\mu}(\psi D_{<\mu})\arrow[rrr,swap,"D_{<\nu}(f^{\mu,\nu})\circ\eta^{\mu,\nu}_{\psi D[<\mu]}"] & & & D_{<\nu}(\psi D_{<\nu})
\end{tikzcd}
\end{equation*}
will thus commute for the unique embedding~$f^{\mu,\nu}$ with range
\begin{equation*}
\Psi^{\mu,\nu}=\left\{\alpha<\psi D_{<\nu}\,\left|\,\pi_\nu(\alpha)\in\rng\left(\eta^{\mu,\nu}_{\psi D[<\nu]}\right)\text{ and }\supp^\nu_{\psi D[<\nu]}(\pi_\nu(\alpha))\subseteq\Psi^{\mu,\nu}\right.\right\}.
\end{equation*}
Let us note that the range of $\eta^{\mu,\nu}_{\psi D[<\nu]}$ is an initial segment of its codomain. Crucially, an induction on~$\alpha<\beta\in\Psi^{\mu,\nu}$ thus yields~$\alpha\in\Psi^{\mu,\nu}$ (as $\alpha'\in\supp^\nu_{\psi D[<\nu]}(\pi_\nu(\alpha))$ implies $\alpha'<\alpha$ by Definition~\ref{def:psi-fp}). It follows that $f^{\mu,\nu}$ is an inclusion of ordinals. Writing~$\iota^\mu:\psi D_{<\mu}\to\sup_{\nu<\lambda}\psi D_{<\nu}=:\Omega$ for the inclusions into the supremum, we obtain an embedding~$\pi$ such that the diagram
\begin{equation*}
\begin{tikzcd}
\psi D_{<\mu}\arrow[rrr,"\iota^\mu"]\arrow[d,swap,"\pi_\mu"] & & & \Omega\arrow[d,dashed,"\pi"]\\
D_{<\mu}(\psi D_{<\mu})\arrow[rrr,swap,"D_{<\lambda}(\iota^\mu)\circ\eta^{\mu,\lambda}_{\psi D[<\mu]}"] & & & D_{<\lambda}(\Omega)
\end{tikzcd}
\end{equation*}
commutes for all~$\mu<\lambda$.

In the following, we show that $\pi:\Omega\to D_{<\lambda}(\Omega)$ validates condition~(\ref{psi-fp2}). The verification of~(\ref{psi-fp1}) is easier and left to the reader. Consider $\alpha<\Omega$ and $\tau\in D_{<\lambda}(\Omega)$ such that $\pi(\alpha)<\tau$ holds for all~$\alpha\in\supp^\lambda_\Omega(\tau)$. We can write $\tau=\eta^{\mu,\lambda}_\Omega(\tau_0)$ for any large enough~$\mu$. Increasing the latter if necessary, we may also assume
\begin{equation*}
\supp^\mu_\Omega(\tau_0)=\supp^\lambda_\Omega(\tau)\subseteq\psi D_{<\mu}=\rng(\iota^\mu),
\end{equation*}
which allows us to write $\tau_0=D_{<\mu}(\iota^\mu)(\tau_1)$. For an arbitrary $\alpha'\in\supp^\mu_{\psi D[<\mu]}(\tau_1)$, we get $\iota^\mu(\alpha')\in\supp^\mu_\Omega(\tau_0)$ and hence
\begin{equation*}
\pi\circ\iota^\mu(\alpha')<\tau=\eta^{\mu,\lambda}_\Omega\circ D_{<\mu}(\iota^\mu)(\tau_1)=D_{<\lambda}(\iota^\mu)\circ\eta^{\mu,\lambda}_{\psi D[<\mu]}(\tau_1).
\end{equation*}
In view of the previous diagram, this yields~$\pi_\mu(\alpha')<\tau_1$. By condition~(\ref{psi-fp2}) for~$D_{<\mu}$, we obtain~$\tau_1\in\rng(\pi_\mu)$, which implies~$\tau\in\rng(\pi)$, as required.
\end{proof}

Our aim is to characterize~$\psi D^\gamma$ by recursion on~$D$. The previous results cover dilators of types~$0,1$ and~$\omega$. For~$D$ of type~$\Omega$, they allow us to focus on the connected component~$E\neq 1$ with~$D=\{D\}^0+E$ (cf.~the observation after Definition~\ref{def:J}). We will use the following variant of separation of variables.

\begin{definition}\label{def:sep+-}
For a connected dilator~$D\neq 1$ and any ordinal~$\gamma$, we define $D^\gamma_-$ and~$D^\gamma_+$ as the dilators with
\begin{align*}
D^\gamma_-(\alpha)&=\{(\sigma;\alpha_0,\ldots,\alpha_n)\in D(\gamma+\alpha)\,|\,\alpha_{i(\sigma)}<\gamma\},\\
D^\gamma_+(\alpha)&=\{(\sigma;\alpha_0,\ldots,\alpha_n)\in D(\gamma+\alpha)\,|\,\alpha_{i(\sigma)}\geq\gamma\},
\end{align*}
where the actions on morphisms and the supports are explained as in Definition~\ref{def:sep-vars} (recall from Definition~\ref{def:import-coeff} that $i(\sigma)$ is the most important argument position).
\end{definition}

Due to Lemma~\ref{lem:import-coeff}, we have $D^\gamma=D^\gamma_-+D^\gamma_+$. Here the summands are uniquely determined by the following property.

\begin{lemma}[$\prs$]\label{lem:D-plus-connected}
If $D$ is connected and different from~$1$, then so is~$D^\gamma_+$.
\end{lemma}
\begin{proof}
The trace of~$D^\gamma_+$ (see Section~\ref{sect:prelim-dil}) can be described as
\begin{multline*}
\tr(D^\gamma_+)=\big\{(\sigma\langle\gamma_0,\ldots,\gamma_{i-1}\rangle,n+1-i)\,:\,(\sigma,n+1)\in\tr(D)\text{ and}\\
\gamma_0<\ldots<\gamma_{i-1}<\gamma\text{ with }i\leq i(\sigma)\big\},
\end{multline*}
where the embeddings~$D^\gamma_+(\alpha)\to D(\gamma+\alpha)$ are then given by
\begin{equation*}
(\sigma\langle\gamma_0,\ldots,\gamma_{i-1}\rangle;\alpha_i,\ldots,\alpha_n;\alpha)\mapsto(\sigma;\gamma_0,\ldots,\gamma_{i-1},\gamma+\alpha_i,\ldots,\gamma+\alpha_n;\gamma+\alpha).
\end{equation*}
From Lemma~\ref{lem:import-coeff} we learn that~$\alpha_{i(\sigma)}<\beta_{i(\tau)}$ implies
\begin{equation*}
(\sigma\langle\gamma_0,\ldots,\gamma_{i-1}\rangle;\alpha_i,\ldots,\alpha_m)<(\tau\langle\delta_0,\ldots,\delta_{j-1}\rangle;\beta_j,\ldots,\beta_n).
\end{equation*}
According to Definition~\ref{def:connected}, we thus have
\begin{equation*}
(\tau\langle\delta_0,\ldots,\delta_{j-1}\rangle,n+1-j)\not\ll(\sigma\langle\gamma_0,\ldots,\gamma_{i-1}\rangle,m+1-i).
\end{equation*}
We similarly get $\not\gg$ and hence~$\equiv$, as needed to show that $D^\gamma_+$ is connected. To see that $D^\gamma_+$ is not constant, it suffices to note that the second component $n+1-i$ of any trace element is positive (since we have $i\leq i(\sigma)\leq n$).
\end{proof}

Due to the lemma, we can iterate the construction.

\begin{definition}
Consider a connected dilator~$D\neq 1$. For ordinals~$\gamma(i)$, we put
\begin{equation*}
D^{\langle\rangle}_+=D\quad\text{and}\quad D^{\langle\gamma(0),\ldots,\gamma(n)\rangle}_+=\left(D^{\langle\gamma(0),\ldots,\gamma(n-1)\rangle}_+\right)^{\gamma(n)}_+.
\end{equation*}
Furthermore, we put $D^{\langle\gamma(0),\ldots,\gamma(n)\rangle}_-=\big(D^{\langle\gamma(0),\ldots,\gamma(n-1)\rangle}_+\big)^{\gamma(n)}_-$.
\end{definition}

Let us point out that we have
\begin{equation*}
D^{\langle\gamma(0),\ldots,\gamma(n)\rangle}_+=D^\gamma_+\quad\text{for}\quad\gamma=\textstyle\sum_{i\leq n}\gamma(i).
\end{equation*}
This means that one can avoid the iterative construction. The latter nevertheless seems natural, e.\,g., because it readily yields
\begin{equation*}
D=\textstyle\sum_{i\leq n} D^{\langle\gamma(0),\ldots,\gamma(i)\rangle}_-+D^{\langle\gamma(0),\ldots,\gamma(n)\rangle}_+.
\end{equation*}
In the context of the following proof, the last summand vanishes for large~$n$. We~note that the next theorem is analogous to Proposition~1.6 of~\cite{increasing-goodstein}.

\begin{theorem}[$\prs$]\label{thm:psi-connected}
For a connected dilator~$D\neq 1$ and any~$\gamma$, we have
\begin{equation*}
\psi D^\gamma=\textstyle\sum_{n<\omega}\gamma(n+1)\quad\text{with}\quad\gamma(0)=\gamma\text{ and }\gamma(n+1)=\psi D^{\langle\gamma(0),\ldots,\gamma(n)\rangle}_-.
\end{equation*}
More precisely, if the sequence of $\psi$-fixed points~$\gamma(n+1)$ exists, then the indicated sum is the $\psi$-fixed point of~$D^\gamma$.
\end{theorem}
\begin{proof}
We abbreviate $\gamma[n]=\langle\gamma(0),\ldots,\gamma(n-1)\rangle$ and write
\begin{equation*}
\pi_n:\gamma(n+1)\to D^{\gamma[n+1]}_-\big(\gamma(n+1)\big)
\end{equation*}
for the embeddings that come with the given $\psi$-fixed points. Note that the sequence of these embeddings is available as in the proof of Proposition~\ref{prop:psi-limit-sum}. By the remark after Definition~\ref{def:sep+-}, we have
\begin{equation*}
\big(D^{\gamma[n]}_+\big)^{\gamma(n)}=D^{\gamma[n+1]}_-+D^{\gamma[n+1]}_+.
\end{equation*}
Writing $\delta(n)=\sum_{n\leq i<\omega}\gamma(i)$, we note~$\gamma(n)+\delta(n+1)=\delta(n)$ to get
\begin{equation*}
D^{\gamma[n]}_+(\delta(n))=D^{\gamma[n+1]}_-(\delta(n+1))+D^{\gamma[n+1]}_+(\delta(n+1)).
\end{equation*}
The inclusions of summands compose into embeddings
\begin{equation*}
\mu_n:D^{\gamma[n+1]}_-(\delta(n+1))\to D^{\gamma[0]}_+(\delta(0))=D^\gamma(\delta(1)).
\end{equation*}
Since $\mu_{n+1}$ factors via~$D^{\gamma[n+1]}_+(\delta(n+1))$, we have $\mu_n(\sigma)<\mu_{n+1}(\tau)$ for all arguments. Writing $\iota_n:\gamma(n+1)\hookrightarrow\delta(n+1)$ for the inclusions, we thus get an embedding
\begin{equation*}
\pi:\delta(1)=\textstyle\sum_{n<\omega}\gamma(n+1)\to D^\gamma(\delta(1))
\end{equation*}
by stipulating
\begin{equation*}
\pi\left(\textstyle\sum_{i<n}\gamma(i+1)+\alpha\right)=\mu_n\circ D^{\gamma[n+1]}_-(\iota_n)\circ\pi_n(\alpha)\qquad\text{for }\alpha<\gamma(n+1).
\end{equation*}
If $\supp^n$ and~$\supp$ are the supports associated with~$D^{\gamma[n+1]}_-$ and~$D^\gamma$, we have
\begin{equation*}
\supp^n_{\delta(n+1)}(\tau)=\left\{\xi\,\left|\,\textstyle\sum_{i<n}\gamma(i+1)+\xi\in\supp_{\delta(1)}(\mu_n(\tau))\right.\right\}.
\end{equation*}
Indeed, this iterates the characterization of supports from Definitions~\ref{def:sep-vars} and~\ref{def:sep+-}.

To show that~$\pi$ satisfies condition~(\ref{psi-fp1}), we assume
\begin{equation*}
\alpha\in\supp_{\delta(1)}(\pi(\beta))\quad\text{with}\quad\beta=\textstyle\sum_{i<n}\gamma(i+1)+\beta'\text{ for }\beta'<\gamma(n+1).
\end{equation*}
If we have $\alpha<\sum_{i<n}\gamma(i+1)$, the desired conclusion~$\alpha<\beta$ is immediate. In the remaining case, we may write $\alpha=\sum_{i<n}\gamma(i+1)+\xi$. We obtain
\begin{equation*}
\xi\in\supp^n_{\delta(n+1)}\big(D^{\gamma[n+1]}_-(\iota_n)\circ\pi_n(\beta')\big)=\supp^n_{\gamma(n+1)}(\pi_n(\beta')).
\end{equation*}
Now condition~(\ref{psi-fp1}) for~$\pi_n$ yields~$\xi<\beta'$ and thus~$\alpha<\beta$.

As a preparation for condition~(\ref{psi-fp2}), we establish
\begin{equation*}
D^\gamma(\delta(1))=\textstyle\bigcup_{n<\omega}\rng(\mu_n).
\end{equation*}
In the construction of~$\mu_n$ above, we have implicitly constructed embeddings
\begin{equation*}
\nu_n:D^{\gamma[n+1]}_+(\delta(n+1))\to D^{\gamma[0]}(\delta(0))=D^\gamma(\delta(1))
\end{equation*}
as well. In view of the paragraph before this proof, these satisfy
\begin{equation*}
D^\gamma(\delta(1))=\textstyle\bigcup_{i\leq n}\rng(\mu_i)\cup\rng(\nu_n).
\end{equation*}
Aiming at a contradiction, we assume there is an element~$\tau\in D^\gamma(\delta(1))$ that can be written as~$\tau=\nu_n(\tau_n)$ for each~$n<\omega$. Write~$\supp^{n,+}$ for the support of~$D^{\gamma[n+1]}_+$. As in the case of~$\mu_n$ above, we obtain
\begin{equation*}
\xi\in\supp^{n,+}_{\delta(n+1)}(\tau_n)\quad\Rightarrow\quad\textstyle\sum_{i<n}\gamma(i+1)+\xi\in\supp_{\delta(1)}(\tau).
\end{equation*}
Since~$\supp_{\delta(1)}(\tau)$ is finite, the conclusion of this implication must fail for large~$n$. But then $\tau_n$ has empty support, which is impossible since~$D^{\gamma[n+1]}_+\neq 1$ is connected.

In order to verify condition~(\ref{psi-fp2}), we now consider an arbitrary~$\tau\in D^\gamma(\delta(1))$ with $\pi(\alpha)<\tau$ for all~$\alpha\in\supp_{\delta(1)}(\tau)$. As we have just seen, we can write $\tau=\mu_n(\tau')$ for a suitable~$n$. Consider the chain of implications
\begin{align*}
\xi\in\supp^n_{\delta(n+1)}(\tau')\quad&\Rightarrow\quad\textstyle\sum_{i<n}\gamma(i+1)+\xi\in\supp_{\delta(1)}(\tau)\\
{}&\Rightarrow\quad\pi\big(\textstyle\sum_{i<n}\gamma(i+1)+\xi\big)<\mu_n(\tau')\quad\Rightarrow\quad\xi<\gamma(n+1).
\end{align*}
Due to the support condition for dilators, this allows us to write $\tau'=D^{\gamma[n+1]}_-(\iota_n)(\tau'')$ with~$\tau''\in D^{\gamma[n+1]}_-(\gamma(n+1))$. By the implications above, $\xi\in\supp^n_{\gamma(n+1)}(\tau'')$ implies
\begin{equation*}
\mu_n\circ D^{\gamma[n+1]}_-(\iota_n)\circ\pi_n(\xi)=\pi\big(\textstyle\sum_{i<n}\gamma(i+1)+\xi\big)<\mu_n\circ D^{\gamma[n+1]}_-(\iota_n)(\tau'')
\end{equation*}
and hence~$\pi_n(\xi)<\tau''$. Now condition~(\ref{psi-fp2}) for~$D^{\gamma[n+1]}_-$ yields~$\tau''\in\rng(\pi_n)$. This readily entails $\tau\in\rng(\pi)$, as required.
\end{proof}

The previous theorem does not correspond to an induction over dilators in the strict sense, where the predecessors of a connected dilator~$D\neq 1$ have the form~$\{D\}^\gamma$ and not~$D^\gamma_-$ (cf.~Section~\ref{sect:prelim-dil}). In the following we show that the modified predecessors still yield a valid induction principle.

\begin{remark}\label{rmk:dil-ind-liberal}
Let $P^\star_\xi(D)$ be explained as in Definition~\ref{def:dil-preds}, except in the case where~$D$ has type~$\Omega$, in which we set
\begin{multline*}
P^\star_\xi(D)=\left\{\left.D_0+E^{\langle\gamma(0),\ldots,\gamma(n)\rangle}_-\,\right|\,\gamma(i)<\xi\text{ for }i\leq n\right\}\\
\text{for }D=D_0+E\text{ with connected }E\neq 1.
\end{multline*}
We claim that Theorems~\ref{thm:guarded-ind} and~\ref{thm:guarded-rec} (guarded induction and recursion on dilators) remain valid when~$P$ is replaced by~$P^*$. To adapt the proofs, it is enough to have a version of Lemma~\ref{lem:sep-pred}. Specifically, we need
\begin{equation*}
E^{\langle\gamma(0),\ldots,\gamma(n)\rangle}_-\left(\omega^\delta\right)<E\left(\omega^\delta\right)\quad\text{for $\gamma(i)<\omega^\delta$ with~$\delta>0$}.
\end{equation*}
For $E_\star=E^{\langle\gamma(0),\ldots,\gamma(n-1)\rangle}_+$ we have~$E_\star(\omega^\delta)=E(\sum_{i<n}\gamma(i)+\omega^\delta)=E(\omega^\delta)$, where the last equality exploits that~$\omega^\delta$ is additively principal. As in the proof of the original lemma (where we needed~$\omega^\delta\geq\omega$), we also have~$(E_\star)^{\gamma(n)}_-(\omega^\delta)<E_\star(\omega^\delta)$.
\end{remark}

In order to connect with the original version of induction on dilators, we first note that $\{D\}^\gamma\leq D^\gamma_-$ holds by Definitions~\ref{def:sep-vars} and~\ref{def:sep+-}. The following converse bound will be needed later.

\begin{lemma}[$\prs$]\label{lem:D-min-sep}
For connected~$D\neq 1$, we have $D^\gamma_-\leq(\{D\}^\gamma)^\gamma$.
\end{lemma}
To avoid misunderstanding, we point out that the superscripts in $(\{D\}^\gamma)^\gamma$ refer to two different constructions, which are explained by Definitions~\ref{def:sep-vars} and~\ref{def:D-gamma}.
\begin{proof}
In view of Definitions~\ref{def:sep-vars} and~\ref{def:sep+-}, we can view~$D^\gamma_-$ and $(\{D\}^\gamma)^\gamma$ as subdilators of $D^\gamma$ and~$D^{\gamma\cdot 2}$. This allows us to define $\eta_\alpha:D^\gamma_-(\alpha)\to(\{D\}^\gamma)^\gamma(\alpha)$ by
\begin{equation*}
(\sigma;\alpha_0,\ldots,\alpha_n;\gamma+\alpha)\mapsto(\sigma;\alpha_0',\ldots,\alpha_n';\gamma\cdot 2+\alpha),
\end{equation*}
where we set
\begin{equation*}
\alpha_i'=\begin{cases}
\alpha_i & \text{if }i\leq i(\sigma),\\
\gamma+\alpha_i & \text{otherwise}.
\end{cases}
\end{equation*}
It is not hard to see that these maps are natural in~$\alpha$. The crucial task is to show that they are order preserving. Let us consider an inequality
\begin{equation*}
(\sigma;\alpha_0,\ldots,\alpha_m;\gamma+\alpha)<(\tau;\beta_0,\ldots,\beta_n;\gamma+\alpha).
\end{equation*}
Our task is to derive
\begin{equation*}
(\sigma;\alpha'_0,\ldots,\alpha'_m;\gamma\cdot 2+\alpha)<(\tau;\beta'_0,\ldots,\beta'_n;\gamma\cdot 2+\alpha),
\end{equation*}
where the~$\alpha_i'$ and~$\beta_j'$ are defined as above. Lemma~\ref{lem:import-coeff} forces~$\alpha_{i(\sigma)}\leq\beta_{i(\tau)}$. If the inequality is strict, we get $\alpha_{i(\sigma)}'<\beta_{i(\tau)}'$, which yields the desired conclusion. Now assume~$\alpha_{i(\sigma)}=\beta_{i(\tau)}$. Crucially, it follows that we have
\begin{equation*}
\alpha_i=\beta_j\quad\Rightarrow\quad\alpha_i'=\beta_j'.
\end{equation*}
The point is that the premise makes~$i\leq i(\sigma)$ equivalent to~$j\leq j(\tau)$. We also get
\begin{equation*}
\alpha_i<\beta_j\quad\Rightarrow\quad\alpha_i'<\beta_j'.
\end{equation*}
Here the premise ensures that~$i(\sigma)<i$ implies~$i(\tau)<j$. We can conclude by the functoriality of dilators (or more explicitly by Section~0.1.1 of~\cite{girard-pi2}).
\end{proof}

To obtain recursive clauses for~$\psi$, we needed to characterize~$\psi D^\gamma$ simultaneously for various~$\gamma$. It will be important to know how the precedessors of $D^\gamma$ can be described in terms of the predecessors of~$D$.

\begin{lemma}[$\prs$]\label{lem:D-gamma-sep}
(a) We have $(\sum_{\mu<\nu}D_\mu)^\gamma=\sum_{\mu<\nu}D_\mu^\gamma$.

(b) If~$D\neq 1$ is connected, $D^\gamma$ has type~$\Omega$ and we get~$\{D^\gamma\}^\delta\leq(\{D\}^{\gamma+\delta})^\gamma$.
\end{lemma}
\begin{proof}
(a) Elements of~$(\sum_{\mu<\nu}D_\mu)^\gamma(\alpha)=(\sum_{\mu<\nu}D_\mu)(\gamma+\alpha)$ have the form
\begin{equation*}
\textstyle\sum_{\mu<\eta}D_\mu+\sigma\quad\text{with}\quad\sigma\in D_\eta(\gamma+\alpha)=D_\eta^\gamma(\alpha).
\end{equation*}
We can identify them with the elements~$\sum_{\mu<\eta}D_\mu^\gamma+\sigma$ of $\sum_{\mu<\nu}D_\mu^\gamma$.

(b) Recall that we have~$D^\gamma=D^\gamma_-+D^\gamma_+$ with connected~$D^\gamma_+\neq 1$, which yields
\begin{equation*}
\{D^\gamma\}^\delta=D^\gamma_-+\{D^\gamma_+\}^\delta.
\end{equation*}
By Lemma~\ref{lem:D-min-sep} and the proof of Lemma~\ref{lem:D-plus-connected}, we have natural embeddings
\begin{equation*}
\eta_\alpha:D^\gamma_-(\alpha)\to(\{D\}^{\gamma})^\gamma(\alpha)\quad\text{and}\quad\xi_\alpha:\{D^\gamma_+\}^\delta(\alpha)\to\{D\}^{\gamma+\delta}(\alpha).
\end{equation*}
We can clearly combine these into a map from $\{D^\gamma\}^\delta(\alpha)$ into~$(\{D\}^{\gamma+\delta})^\gamma$. Crucially, this map is an embedding because the leading argument of any value of~$\xi_\alpha$ must have the form~$\gamma+\delta_0$ with~$\delta_0<\delta$.
\end{proof}

\section{Bounding~$\psi$ by~$\mathbb J$}\label{sect:comparison}

In this section, we show how $\psi$-fixed points can be bounded by values of the functor~$\mathbb J$ due to P\"appinghaus (see the previous two sections for definitions of these objects). In particular, we prove that well-founded $\psi$-fixed points exist when~$\mathbb J$ is total. This completes the proof of Theorem~\ref{thm:J-Pi11CA} from the introduction.

It will be convenient to have a variant~$\mathbb J'$ of~$\mathbb J$ for which the recursive clause for dilators of type~$\Omega$ reads
\begin{equation*}
\mathbb J'(D,\gamma)=\alpha+\mathbb J'(\{D\}^\alpha,\gamma)\quad\text{with}\quad\alpha=\mathbb J'(\{D\}^\omega,\gamma).
\end{equation*}
Note that this has~$\{D\}^\omega$ where the clause for~$\mathbb J$ has~$\{D\}^0$. As before, the following provides a construction in our base theory.

\begin{definition}
The relation $\mathbb J'(D,\gamma)\simeq_\xi^\eta\delta$ is defined by the same recursive clauses as the relation $\mathbb J(D,\gamma)\simeq_\xi^\eta\delta$ from Definition~\ref{def:J}, except that clause~(iv) of the latter is replaced by the following:
\begin{enumerate}[label=(\roman*')]\setcounter{enumi}{3}
\item For~$D$ of type~$\Omega$, we have $\mathbb J'(D,\gamma)\simeq_\xi^\eta\delta$ precisely if there are ordinals~$\alpha,\beta$ with $\delta=\alpha+\beta$ and~$\mathbb J'(\{D\}^\omega,\gamma)\simeq_\xi^\eta\alpha$ as well as $\mathbb J'(\{D\}^\alpha,\gamma)\simeq_\xi^\eta\beta$.
\end{enumerate}
As before, we write $\mathbb J'(D,\gamma)\simeq\delta$ to express that $\mathbb J'(D,\gamma)\simeq_\xi^\eta\delta$ holds for some~$\eta,\xi$.
\end{definition}

All results of Section~\ref{sect:J} apply to~$\mathbb J'$ as well (essentially without changes in the proofs). In the following, we write $D\cdot\mu$ for~$\sum_{\alpha<\mu}D$.

\begin{lemma}[$\atrs$]
For $\gamma\geq\omega$ we have $\mathbb J'(D,\gamma)\lesssim\mathbb J(D\cdot 8,\gamma)$.
\end{lemma}
Here the symbol~$\lesssim$ is used as in Section~\ref{sect:J}. More explicitly, the lemma says that $\mathbb J(D\cdot 8,\gamma)\simeq\nu$ entails the existence if a~$\mu\leq\nu$ with $\mathbb J'(D,\gamma)\simeq\mu$. When the same holds with~$\simeq^\eta_\xi$ at the place of~$\simeq$, we write $\lesssim^\eta_\xi$ rather than~$\lesssim$. The lemma implies that~$\mathbb J'$ is total if the same holds for~$\mathbb J$.
\begin{proof}
For fixed~$\eta$ and $\xi$, we establish $\mathbb J'(D,\gamma)\lesssim^\eta_\xi\mathbb J(D\cdot 8,\gamma)$ by guarded induction on~$D$. The purpose of~$\eta$ and~$\xi$ is to make the induction statement primitive recursive. In view of Lemma~\ref{lem:J-indep}, however, it is enough to show $\mathbb J'(D,\gamma)\lesssim\mathbb J(D\cdot 8,\gamma)$ in the induction step (i.\,e., the bounds~$\eta$ and~$\xi$ come for free).

From Lemma~\ref{lem:J-nat-trans} we know that $D\leq E$ implies~$\mathbb J(D,\gamma)\lesssim\mathbb J(E,\gamma)$. This covers the induction step for dilators of the types~$1$ and~$\omega$ (due to~$D\cdot 8+1\leq (D+1)\cdot 8$). The case of type~$0$ is immediate.

It remains to consider a dilator~$D$ of type~$\Omega$. As preparation, we establish that we have~$\{D\}^\delta\cdot 8\leq\{D\}^{\delta\cdot 8}$. Elements of~$(\{D\}^\delta\cdot 8)(\alpha)$ correspond to pairs~$(i,\tau)$ with~$i<8$ and~$\tau\in\{D\}^\delta(\alpha)$, compared lexicographically. Writing~$\tau$ as in Definition~\ref{def:sep-vars}, we map into~$\{D\}^{\delta\cdot 8}(\alpha)$ by stipulating
\begin{equation*}
\big(i,(\sigma;\alpha_0,\ldots,\alpha_n)\big)\mapsto(\sigma;\alpha'_0,\ldots,\alpha'_n)
\end{equation*}
with $\alpha'_j=(i,\alpha_j)\in\delta\cdot 8$ for~$i\leq i(\sigma)$ and~$\alpha'_j=\alpha_j$ otherwise. This map is order preserving since comparisons between values are dominated by the component~$i$ of the crucial argument~$\alpha'_{i(\sigma)}=(i,\alpha_{i(\sigma)})$. Clearly, the construction is natural in~$\alpha$.

By the results of Section~\ref{sect:J}, we have
\begin{equation*}
\mathbb J(D\cdot n,\gamma)\cdot 2\lesssim\mathbb J(D,\mathbb J(D\cdot n,\gamma))\simeq\mathbb J(D\cdot(n+1),\gamma).
\end{equation*}
In view of~$\omega\leq\gamma=\mathbb J(D\cdot 0,\gamma)$, this implies~$\omega\cdot 8\lesssim\mathbb J(D\cdot 3,\gamma)$. The previous paragraph and Lemma~\ref{lem:Omega-nat-trans} yield~$\{D\}^\omega\cdot 8\leq\{D\}^{\mathbb J(D\cdot 3,\gamma)}$. Inductively, we get
\begin{multline*}
\delta:=\mathbb J'\big(\{D\}^\omega,\gamma\big)\lesssim\mathbb J\big(\{D\}^\omega\cdot 8,\gamma\big)\lesssim\mathbb J\big(\{D\}^{\mathbb J(D\cdot 3,\gamma)},\mathbb J(D\cdot 3,\gamma)\big)\\
{}\lesssim\mathbb J(D,\mathbb J(D\cdot 3,\gamma))\simeq\mathbb J(D\cdot 4,\gamma).
\end{multline*}
As above, this implies~$\delta\cdot 8\lesssim\mathbb J(D\cdot 7,\gamma)$. Note that we have~$\{D\cdot 8\}^\delta=D\cdot 7+\{D\}^\delta$. Due to the induction hypothesis, we get
\begin{multline*}
\mathbb J'(D,\gamma)\simeq\delta+\mathbb J'\big(\{D\}^\delta,\gamma\big)\lesssim\delta+\mathbb J\big(\{D\}^\delta\cdot 8,\gamma\big)\\
{}\lesssim\mathbb J(D\cdot 7,\gamma)+\mathbb J\big(\{D\}^{\mathbb J(D\cdot 7,\gamma)},\gamma\big)\lesssim\mathbb J(D\cdot 8,\gamma),
\end{multline*}
as needed to complete the induction step.
\end{proof}

The next proof would not go through with~$\mathbb J$ in place of~$\mathbb J'$.

\begin{lemma}[$\atrs$]\label{lem:D'-robust}
For $n<\omega\leq\gamma$ we have $\mathbb J'(D^n,\gamma)\simeq\mathbb J'(D,\gamma)$.
\end{lemma}
\begin{proof}
We have~$\gtrsim$ since~$D$ embeds into~$D^n$ (see Lemma~\ref{lem:J-nat-trans}). For the converse inequality, we argue by guarded induction on~$D$. As in the previous proof, it is straightforward to replace~$\lesssim$ by~$\lesssim^\eta_\xi$ for suitable bounds, so that the induction statement becomes primitive recursive. We have~$D^n=D$ when~$D$ is constant. Together with Lemma~\ref{lem:D-gamma-sep}, this covers the induction step for dilators of the types~$0,1$ and~$\omega$. When~$D$ has type~$\Omega$, we write~$D=D_0+E$ with connected~$E\neq 1$. Let us observe that we have~$\{D^n\}^\delta=D_0^n+\{E^n\}^\delta$. Part~(b) of Lemma~\ref{lem:D-gamma-sep} yields
\begin{equation*}
\{E^n\}^\delta\leq(\{E\}^{n+\delta})^n=(\{E\}^\delta)^n\quad\text{for }\delta\geq\omega.
\end{equation*}
Inductively, we get
\begin{equation*}
\alpha:=\mathbb J'\big(\{E^n\}^\omega,\gamma\big)\lesssim\mathbb J'\big((\{E\}^\omega)^n,\gamma\big)\simeq\mathbb J'\big(\{E\}^\omega,\gamma\big)=:\beta.
\end{equation*}
As we have~$\alpha\geq\gamma\geq\omega$, another application of the induction hypothesis yields
\begin{align*}
\mathbb J'(E^n,\gamma)&\simeq\alpha+\mathbb J'\big(\{E^n\}^\alpha,\gamma\big)\lesssim\beta+\mathbb J'\big(\{E^n\}^\beta,\gamma\big)\\
{}&\lesssim\beta+\mathbb J'\big((\{E\}^\beta)^n,\gamma\big)\lesssim\beta+\mathbb J'\big(\{E\}^\beta,\gamma\big)\simeq\mathbb J'(E,\gamma).
\end{align*}
For $D=D_0+E$, the result is readily derived by Lemma~\ref{lem:J-nat-trans}(b).
\end{proof}

Our next aim is to show that values of the form~$\mathbb J'(\omega\circ(D+1),\gamma)$ have very good closure properties. Here $\omega$ refers to the dilator with
\begin{equation*}
\omega(\alpha)=\left\{\omega^{\alpha_0}+\ldots+\omega^{\alpha_{n-1}}\,|\,\alpha>\alpha_0\geq\ldots\geq\alpha_{n-1}\right\}.
\end{equation*}
One should think of Cantor normal forms, though officially the elements of~$\omega(\alpha)$ are formal expressions, which are compared lexicographically. When~$f:\alpha\to\beta$ is an embedding, $\omega(f):\omega(\alpha)\to\omega(\beta)$ is defined by pointwise application of~$f$ to the exponents. The support of $\omega^{\alpha_0}+\ldots+\omega^{\alpha_{n-1}}$ is given by~$\{\alpha_0,\ldots,\alpha_{n-1}\}$.

\begin{lemma}[$\prs$]\label{lem:omega-preds}
For any dilator~$D$ of type~$\Omega$, we have the following:
\begin{enumerate}[label=(\alph*)]
\item There is an~$n<\omega$ such that $\{D\}^\gamma+1\leq(\{D\}^{\gamma+1})^n$ holds for all~$\gamma\geq n$.
\item The dilator~$\omega\circ D$ is of type~$\Omega$ and we have $\omega\circ\{D\}^\gamma\leq\{\omega\circ D\}^\gamma$.
\end{enumerate}
\end{lemma}
\begin{proof}
(a) It is straightforward to reduce to the case where~$D\neq 1$ is connected. Here we fix an element~$(\sigma,n)\in\tr(D)$. Via the obvious embedding of~$\gamma+\alpha$ into $\gamma+1+n+\alpha$, we can realize~$\{D\}^\gamma(\alpha)\subseteq D(\gamma+\alpha)$ as a suborder of
\begin{equation*}
(\{D\}^{\gamma+1})^n(\alpha)\subseteq D(\gamma+1+n+\alpha).
\end{equation*}
Given~$\gamma\geq n$, we can consider
\begin{equation*}
(\sigma;0,\ldots,i(\sigma)-1,\gamma,\gamma+1,\ldots,\gamma+n-i(\sigma))\in(\{D\}^{\gamma+1})^n(\alpha).
\end{equation*}
This element lies above all of~$\{D\}^\gamma(\alpha)$, since we have~$\gamma$ at the most important argument position~$i(\sigma)$.

(b) Write $D=D_0+D_1$ for connected~$D_1\neq 1$. It is not hard to see that there is a dilator~$E$ with $\omega\circ D=\omega\circ D_0+E$. Indeed, a term $\omega^{\sigma(0)}+\ldots+\omega^{\sigma(n-1)}\in\omega\circ D(\alpha)$ lies in~$E(\alpha)$ precisely when we have~$n>0$ and~$\sigma(0)=D_0(\alpha)+\rho$ with $\rho\in D_1(\alpha)$. Given that~$D_1$ is connected, it is straightforward to conclude that the same holds for~$E$. Thus~$\omega\circ D$ is of type~$\Omega$ and we have
\begin{equation*}
\{\omega\circ D\}^\gamma=\omega\circ D_0+\{E\}^\gamma.
\end{equation*}
In view of~$\{D\}^\gamma=D_0+\{D_1\}^\gamma$, we also find a dilator~$E'$ with
\begin{equation*}
\omega\circ\{D\}^\gamma=\omega\circ D_0+E'.
\end{equation*}
Consider an arbitrary element
\begin{equation*}
\tau=\omega^{\tau(0)}+\ldots+\omega^{\tau(n)}\in E'(\alpha)\subseteq\omega(D_0(\alpha)+D_1(\gamma+\alpha))\subseteq\omega\circ D(\gamma+\alpha).
\end{equation*}
As above, we can write $\tau(0)=D_0(\alpha)+\rho(0)$ with $\rho(0)\in\{D_1\}^\gamma(\alpha)$. The most important argument of~$\tau$ coincides with the most important argument of~$\rho(0)$, which comes from~$\gamma$. If $\tau(i)$ has some argument~$\gamma_0<\gamma$, we can write $\tau(i)=D_0(\alpha)+\rho(i)$. Now~$\gamma_0$ is bounded by the most important argument of~$\rho(i)$. The latter must be bounded by the most important argument of~$\rho(0)$, as we have~$\rho(i)\leq\rho(0)$. So in~$\tau$, the most important argument is the biggest argument from~$\gamma$. We can thus identify~$\tau$ with an element of $\{E\}^\gamma(\alpha)\subseteq\omega\circ D(\gamma+\alpha)$. This shows that we have~$E'\leq\{E\}^\gamma$ and hence~$\omega\circ\{D\}^\gamma\leq\{\omega\circ D\}^\gamma$.
\end{proof}

The following values will be crucial for the connection with our $\psi$-fixed points.

\begin{definition}
We use $\mathbb J^+(D,\gamma)$ as an abbreviation for~$\mathbb J'(\omega\circ(D+1),\gamma)$.
\end{definition}

As promised, we now establish a strong closure property.

\begin{proposition}[$\atrs$]
Assume we have $\mathbb J^+(D,\gamma)\simeq\delta$ for a dilator~$D$ of type~$\Omega$ and some~$\gamma\geq\omega$. Then~$\delta>\gamma$ is additively principal and we have~$\mathbb J^+(\{D\}^\alpha,\alpha)<\delta$ for all~$\alpha<\delta$ (which means in particular that the values $\mathbb J^+(\{D\}^\alpha,\alpha)$ exist).
\end{proposition}
\begin{proof}
We observe~$\omega\circ(D+1)=\sum_{N<\omega}\omega\circ D$ and conclude
\begin{equation*}
\delta=\textstyle\sup_{N<\omega}\delta_N\quad\text{with}\quad\mathbb J'\big(\textstyle\sum_{i<N}\omega\circ D,\gamma\big)\simeq\delta_N.
\end{equation*}
By the previous lemma, $\omega\circ D$ has type~$\Omega$. Using Lemma~\ref{lem:J-nat-trans}(b), we get
\begin{multline}\label{eq:delta-n}
\mathbb J'(\omega\circ D,\delta_N)\simeq\delta_{N+1}=\eta+\xi\quad\text{with}\\
\mathbb J'(\{\omega\circ D\}^\omega,\delta_N)\simeq\eta\text{ and }\mathbb J'(\{\omega\circ D\}^\eta,\delta_N)\simeq\xi.
\end{multline}
Here we have~$\delta_N\leq\eta,\xi$ and hence~$\delta_N\cdot 2\leq\delta_{N+1}$, which entails that~$\delta$ is additively principal. By Lemma~\ref{lem:J-properties}(b) we have~$\delta>\gamma$.

Let us now show $\mathbb J^+(\{D\}^\alpha,\alpha)<\delta$ for given~$\alpha<\delta$. We may assume~$\omega\leq\alpha<\delta_N$ for a suitable~$N<\omega$. Pick an~$n<\omega$ that validates part~(a) of the previous lemma. In view of~$\omega\circ(\{D\}^\alpha+1)\leq\omega\circ(\{D\}^{\alpha+1})^n=(\omega\circ\{D\}^{\alpha+1})^n$, Lemma~\ref{lem:D'-robust} yields
\begin{equation*}
\mathbb J^+\big(\{D\}^\alpha,\alpha\big)\lesssim\mathbb J'\big(\omega\circ\{D\}^{\alpha+1},\alpha\big)\lesssim\mathbb J'\big(\{\omega\circ D\}^{\alpha+1},\alpha\big).
\end{equation*}
Let us write~$\delta_{N+1}=\eta+\xi$ as in~(\ref{eq:delta-n}). We get~$\eta\geq\delta_N\geq\alpha+1$ and thus
\begin{equation*}
\mathbb J^+\big(\{D\}^\alpha,\alpha\big)\lesssim\mathbb J'\big(\{\omega\circ D\}^\eta,\delta_N\big)\simeq\xi\leq\delta_{N+1}<\delta,
\end{equation*}
which proves the open claim.
\end{proof}

Finally, we can establish the desired bound on~$\psi$-fixed points.

\begin{theorem}[$\atrs$]
We have $\gamma+\psi D^\gamma\lesssim\mathbb J^+(D,\gamma)$ for any dilator~$D$ and any infinite ordinal~$\gamma$.
\end{theorem}
\begin{proof}
By guarded induction on~$D$, we show that the claim holds whenever $\mathbb J^+(D,\gamma)$ lies below some fixed bound (which bounds the quantification over~$\gamma$ and makes the induction statement primitive recursive; cf.~also Lemma~\ref{lem:psi-prim-rec}). For~$D=0$ we have~$D^\gamma=0$ and thus~$\gamma+\psi D^\gamma=\gamma\leq\mathbb J^+(D,\gamma)$, due to Lemmas~\ref{lem:J-properties}(a) and~\ref{lem:psi-constant}. When our dilator has the form~$D+1$, we note that $\psi(D+1)^\gamma=\psi D^\gamma+1$ is a special case of Proposition~\ref{prop:psi-sum}. Together with $\omega\circ(D+1)+1\leq\omega\circ(D+2)$, this accounts for the induction step. For a dilator of type~$\omega$, it suffices to invoke Proposition~\ref{prop:psi-limit-sum}, since we can use Lemma~\ref{lem:J-nat-trans}(a) to get
\begin{equation*}
\textstyle\sup_{\nu<\lambda}\mathbb J^+\big(\textstyle\sum_{\mu<\nu}D_\mu,\gamma\big)\lesssim\mathbb J^+\big(\textstyle\sum_{\mu<\lambda}D_\mu,\gamma\big),
\end{equation*}
though the inequality may be strict.

Finally, we assume~$D=D_0+E$ is a dilator of type~$\Omega$, where~$E\neq 1$ is connected. For a suitable~$n<\omega$, we have a natural embedding~$2+E\leq E^n$, which yields
\begin{equation*}
\omega\circ(D_0+1)+\omega\circ(E+1)\leq\omega\circ(D_0+E^n+1)\leq\big(\omega\circ(D+1)\big)^n.
\end{equation*}
Using Lemmas~\ref{lem:J-nat-trans} and~\ref{lem:D'-robust}, we thus get
\begin{equation*}
\mathbb J^+\big(E,\mathbb J^+(D_0,\gamma)\big)\lesssim\mathbb J'\big((\omega\circ(D+1))^n,\gamma\big)\simeq\mathbb J^+\big(D,\gamma\big).
\end{equation*}
The induction hypothesis yields~$\delta:=\gamma+\psi D_0^\gamma\lesssim\mathbb J^+(D_0,\gamma)$, while Proposition~\ref{prop:psi-sum} provides~$\gamma+\psi D^\gamma=\delta+\psi E^\delta$. In order to complete the induction step, it is thus enough to prove~$\delta+\psi E^\delta\lesssim\mathbb J^+(E,\delta)$.

According to Theorem~\ref{thm:psi-connected}, we can write
\begin{equation*}
\delta+\psi E^\delta=\textstyle\sum_{n<\omega}\delta(n)\quad\text{with}\quad\delta(0)=\delta\text{ and }\delta(n+1)=\psi E^{\langle\delta(0),\ldots,\delta(n)\rangle}_-.
\end{equation*}
As~$\mathbb J^+(E,\delta)$ is additively principal by the previous proposition, it suffices to show that we have~$\delta(n)<\mathbb J^+(E,\delta)$ for all~$n<\omega$. Concerning the formalization in~$\atrs$, we point out that the sequence~$n\mapsto\delta(n)$ can then be constructed by primitive recursion, since~$\psi$-fixed points have computable representations (see the paragraph after Lemma~\ref{lem:psi-unique}) and since order types of well orders can be computed below a given bound (by Lemma~\ref{lem:otp-restr}). In view of~$1\leq\omega\circ(E+1)$, we get $\delta(0)<\mathbb J^+(E,\delta)$ from Lemma~\ref{lem:J-properties}(b). Inductively, we now assume $\delta[n+1]:=\sum_{i<n+1}\delta(i)<\mathbb J^+(E,\delta)$. In view of $E^{\langle\delta(0),\ldots,\delta(n-1)\rangle}_+\leq E^{\delta[n]}$, Lemmas~\ref{lem:Omega-nat-trans},~\ref{lem:D-gamma-sep} and~\ref{lem:D-min-sep} yield
\begin{equation*}
E^{\langle\delta(0),\ldots,\delta(n)\rangle}_-=(E^{\langle\delta(0),\ldots,\delta(n-1)\rangle}_+)^{\delta(n)}_-\leq(\{E^{\delta[n]}\}^{\delta(n)})^{\delta(n)}\leq(\{E\}^{\delta[n+1]})^{\delta[n+1]}.
\end{equation*}
By the induction hypothesis and the previous proposition, we can infer
\begin{equation*}
\delta(n+1)\lesssim\psi(\{E\}^{\delta[n+1]})^{\delta[n+1]}\lesssim\mathbb J^+(\{E\}^{\delta[n+1]},\delta[n+1])<\mathbb J^+(E,\delta),
\end{equation*}
which concludes the induction step.
\end{proof}

The following completes our proof of Theorem~\ref{thm:J-Pi11CA} from the introduction (with the converse implication provided by Theorem~\ref{thm:Pi11CA-to-J}).

\begin{corollary}[$\atrs$]\label{cor:main-reversal}
If~$\mathbb J$ is total (in the sense of Theorem~\ref{thm:Pi11CA-to-J}), then the principle of~$\Pi^1_1$-comprehension is valid.
\end{corollary}
\begin{proof}
By the previous theorem, the totality of~$\mathbb J$ entails that~$\psi D\lesssim\psi D^\omega$ exists as an ordinal for each dilator~$D$. The latter entails $\Pi^1_1$-comprehension by Corollary~4.4 of~\cite{freund-rathjen-iterated-Pi11} (with~$\nu=1$; cf.~the paragraph after Lemma~\ref{lem:psi-unique}).
\end{proof}

The picture is completed by the following remark, which shows how~$\mathbb J$ can be bounded in terms of~$\psi$. Since this bound is not needed for our metamathematical result (i.\,e., for Theorem~\ref{thm:J-Pi11CA}), we omit some details. Note, however, that the bound yields an alternative proof for Theorem~\ref{thm:Pi11CA-to-J} (since~$\Pi^1_1$-comprehension entails the existence of~$\psi$-fixed points by~\cite{freund-rathjen-iterated-Pi11}).

\begin{remark}
Let us write~$\mathsf{Id}$ for the identity dilator, so that the dilator~$\mathsf{Id}\cdot\omega\cdot D$ maps~$\alpha$ to~$\alpha\cdot\omega\cdot D(\alpha)$. We use guarded induction on~$D$ to establish
\begin{equation*}
\mathbb J(D,\gamma)\lesssim\psi(\mathsf{Id}\cdot\omega\cdot D)^\gamma\quad\text{for $D\neq 0$ and additively principal~$\gamma\geq\omega$}.
\end{equation*}
The induction step is immediate when~$D$ is equal to~$1$ or a sum of limit type. For connected~$D\neq 1$, we inductively get
\begin{equation*}
\mathbb J(D,\gamma)\simeq\gamma+\mathbb J(\{D\}^\gamma,\gamma)\lesssim\gamma+\psi(\mathsf{Id}\cdot\omega\cdot\{D\}^\gamma)^\gamma.
\end{equation*}
One checks that $\mathsf{Id}\cdot\omega\cdot D$ is connected with~$\mathsf{Id}\cdot\omega\cdot\{D\}^\gamma\leq\{\mathsf{Id}\cdot\omega\cdot D\}^\gamma$ (cf.~Lemma~\ref{lem:omega-preds}). For any connected~$E\neq 1$ and additively principal~$\gamma$, we also have~$(\{E\}^\gamma)^\gamma\leq E_-^\gamma$ (cf.~Lemma~\ref{lem:D-min-sep}). Using Theorem~\ref{thm:psi-connected}, we obtain
\begin{equation*}
\mathbb J(D,\gamma)\lesssim\gamma+\psi(\mathsf{Id}\cdot\omega\cdot D)^\gamma_-\lesssim\gamma+\psi(\mathsf{Id}\cdot\omega\cdot D)^\gamma.
\end{equation*}
It is now enough to show that~$\psi(\mathsf{Id}\cdot\omega\cdot D)^\gamma>\gamma$ is additively principal. This can be established by another induction on~$D\neq 0$ (the point being that $\mathsf{Id}\cdot\omega=\sum_{n<\omega}\mathsf{Id}$ yields~$\psi(\mathsf{Id}\cdot\omega)^\gamma=\sup_{n<\omega}\psi(\mathsf{Id}\cdot n)^\gamma$ for~$D=1$). Finally, when the claim holds for~$D$ and~$E$, it is readily derived for~$D+E$ (again as~$\psi(\mathsf{Id}\cdot\omega\cdot D)^\gamma$ is additively principal). 
\end{remark}

\bibliographystyle{amsplain}
\bibliography{Dilator-Induction}

\end{document}